\documentclass[a4paper,11pt,leqno,english]{amsart}

\usepackage{enumerate}
\usepackage{amssymb,latexsym}
\usepackage[T1]{fontenc}
\usepackage[centering]{geometry}
\usepackage{url}
\usepackage[french, main=english]{babel}
\usepackage[utf8]{inputenc}
\usepackage[all]{xy}
\usepackage{mathrsfs}
\usepackage[dvipsnames]{xcolor}
\usepackage{comment}
\definecolor{violet}{rgb}{0.0,0.2,0.7}
\definecolor{rouge2}{rgb}{0.8,0.0,0.2}
\usepackage{tikz}
\usepackage{empheq}
\usepackage{tikz-cd}
\usetikzlibrary{matrix,arrows,decorations.pathmorphing}
\usepackage[backref=page]{hyperref}
\usepackage[nameinlink, capitalize]{cleveref}
\hypersetup{
    bookmarks=true,         
    unicode=false,          
    pdftoolbar=true,        
    pdfmenubar=true,        
    pdffitwindow=false,     
    pdfstartview={FitH},    
    pdftitle={},            
    pdfauthor={},           
    colorlinks=true,        
    linkcolor=rouge2,       
    citecolor=violet,       
    filecolor=black,        
    urlcolor=cyan}          
\usepackage{enumitem, xspace}
\usepackage{appendix}

\setlist[enumerate]{
  label=(\thethm.\arabic*),
  itemsep=1ex
}

\setlist[itemize]{
  leftmargin=*,
  topsep=1ex,
  itemsep=1ex,
  label=$\circ$
}

\theoremstyle{plain}    
\newtheorem{thm}{Theorem}[section]
\theoremstyle{plain} 
\newtheorem{bigthm}{Theorem}

 \numberwithin{equation}{section} 
 \numberwithin{figure}{section} 
 \newtheorem{cor}[thm]{Corollary} 
 
 \theoremstyle{plain}    
 \newtheorem{prop}[thm]{Proposition} 
 \theoremstyle{plain}    
 \newtheorem{lem}[thm]{Lemma} 
 \theoremstyle{remark}
  \newtheorem{claim}[thm]{Claim} 
 \theoremstyle{remark}
 \newtheorem{rem}[thm]{Remark}
 \newtheorem*{rem-plain}{Remark}
\newtheorem{exa}[thm]{Example}
\newtheorem{setup}[thm]{Setup}
\theoremstyle{plain}

\theoremstyle{definition}
\newtheorem{defi}[thm]{Definition}

\newcommand{\inv}{^{-1}}
\newcommand{\from}{\colon}

\newcommand{\imp}{\Rightarrow}
\newcommand{\lto}{\longrightarrow}

\newcommand{\inj}{\hookrightarrow}

\newcommand{\bij}{\xrightarrow{\,\smash{\raisebox{-.5ex}{\ensuremath{\scriptstyle\sim}}}\,}}
\newcommand{\isom}{\cong}
\newcommand{\defn}{\coloneqq}
\newcommand{\ndef}{\eqqcolon}
\newcommand{\tensor}{\otimes}
\newcommand{\id}{\mathrm{id}}

\newcommand{\wt}{\widetilde}

\newcommand{\wh}{\widehat}

\newcommand{\dual}{^{\smash{\scalebox{.7}[1.4]{\rotatebox{90}{\textup\guilsinglleft}}}}}
\newcommand{\ddual}{^{\smash{\scalebox{.7}[1.4]{\rotatebox{90}{\textup\guilsinglleft} \hspace{-.5em} \rotatebox{90}{\textup\guilsinglleft}}}}}
\newcommand{\acts}{\ \rotatebox[origin=c]{-90}{\ensuremath{\circlearrowleft}}\ }

\newcommand{\factor}[2]{\left. \raise 2pt\hbox{$#1$} \right/\hskip -2pt \raise -2pt\hbox{$#2$}}

\newdir{ ir}{{}*!/-5pt/@^{(}} 
\newdir{ il}{{}*!/-5pt/@_{(}} 

\DeclareMathOperator{\img}{im}

\DeclareMathOperator{\Aut}{Aut}

\DeclareMathOperator{\Gal}{Gal}

\DeclareMathOperator{\Exc}{Exc}

\DeclareMathOperator{\supp}{supp}

\DeclareMathOperator{\Ric}{Ric}

\DeclareMathOperator{\End}{End}

\newcommand{\orb}{\mathrm{orb}}

\newcommand{\set}[1]{\left\{ #1 \right\}}

\def\rd#1.{\lfloor{#1}\rfloor}
\def\rp#1.{\lceil{#1}\rceil}
\def\tw#1.{\langle{#1}\rangle}

\newcommand{\la}{\langle}
\newcommand{\ra}{\rangle}

\renewcommand{\O}[1]{\mathscr{O}_{#1}}
\newcommand{\Omegap}[2]{\Omega_{#1}^{#2}}
\newcommand{\Omegar}[2]{\Omega_{#1}^{[#2]}}
\newcommand{\T}[1]{\mathscr{T}_{#1}}
\newcommand{\can}[1]{\omega_{#1}}

\newcommand{\reg}[1]{{#1}_{\mathrm{reg}}}
\newcommand{\sing}[1]{{#1}_{\mathrm{sg}}}

\newcommand{\codim}[2]{\mathrm{codim}_{#1}(#2)}

\newcommand{\cc}[2]{\mathrm{c}_{#1} \! \left( #2 \right)}
\newcommand{\cpc}[3]{\mathrm{c}_{#1}^{#2} \! \left( #3 \right)}
\newcommand{\ccorb}[2]{\mathrm{\tilde c}_{#1} \! \left( #2 \right)}
\newcommand{\cpcorb}[3]{\mathrm{\tilde c}_{#1}^{#2} \! \left( #3 \right)}
\newcommand{\chorb}[2]{\mathrm{c}_{#1}^{\mathrm{orb}} \! \left( #2 \right)}

\newcommand{\piorb}[1]{\pi_1^{\mathrm{orb}} \!\left( #1 \right)}

\def\Hnought#1.#2.{\mathit{\Gamma} \!\left( #1, #2 \right)}
\def\HH#1.#2.#3.{\mathrm{H}^{#1} \!\left( #2, #3 \right)}
\def\HHdR#1.#2.{\mathrm{H}^{#1}_{\mathrm{dR}} \!\left( #2 \right)}
\def\HHdRc#1.#2.{\mathrm{H}^{#1}_{\mathrm{dR},\, \mathrm{c}} \!\left( #2 \right)}
\def\HHs#1.#2.{\mathrm{H}^{#1} \!\left( #2 \right)}
\def\HHsc#1.#2.{\mathrm{H}^{#1}_{\mathrm{c}} \!\left( #2 \right)}
\def\hh#1.#2.#3.{h^{#1} \!\left( #2, #3 \right)}
\def\RR#1.#2.#3.{R^{#1} #2_* #3}
\def\HHc#1.#2.#3.{\mathrm{H}_{\mathrm{c}}^{#1} \!\left( #2, #3 \right)}
\def\Hh#1.#2.#3.{\mathrm{H}_{#1} \!\left( #2, #3 \right)}
\def\Hhs#1.#2.{\mathrm{H}_{#1} \!\left( #2 \right)}
\def\Hom#1.#2.{\mathrm{Hom} \!\left( #1, #2 \right)}
\def\sHom#1.#2.{\mathscr{H}\!om \!\left( #1, #2 \right)}
\def\Ext#1.#2.#3.{\mathrm{Ext}^{#1} \!\left( #2, #3 \right)}
\def\sExt#1.#2.#3.{\mathscr{E}\!xt^{#1} \!\left( #2, #3 \right)}
\def\Link#1.#2.{\mathrm{Link} \!\left( #1, #2 \right)}

\newcommand{\B}{\ensuremath{\mathbb{B}}}
\newcommand{\C}{\ensuremath{\mathbb{C}}}

\newcommand{\PP}{\ensuremath{\mathbb{P}}}
\newcommand{\Q}{\ensuremath{\mathbb{Q}}}
\newcommand{\R}{\ensuremath{\mathbb{R}}}

\newcommand{\Z}{\ensuremath{\mathbb{Z}}}

\newcommand{\KE}{K{\"{a}}hler--Einstein\xspace}

\newcommand{\kahler}{K{\"{a}}hler\xspace}

\newcommand{\qe}{quasi-\'etale\xspace}

\newcommand{\canmod}{\mathrm{can}}

\newcommand{\cC}{{\mathcal{C}}}

\newcommand{\cN}{{\mathcal{N}}}
\newcommand{\cO}{{\mathcal{O}}}

\newcommand{\cT}{{\mathcal{T}}}

\def\1{\bold{1}}

\newcommand{\fE}{{\mathfrak{E}}}

\newcommand{\Xc}{X^{\circ}}

\renewcommand{\a}{\alpha}

\newcommand{\ab}{{\alpha\beta}}
\newcommand{\bg}{{\beta\gamma}}

\newcommand{\om}{\omega}

\newcommand{\vp}{\varphi}

\newcommand{\ep}{\varepsilon}
\renewcommand{\theta}{\vartheta}

\newcommand{\NS}[1]{\operatorname{N}^1(#1)_\Q}
\newcommand{\Chow}[2]{\operatorname{A}_{#2}(#1)}

\newcommand{\sE}{\mathscr E}
\newcommand{\sF}{\mathscr F}

\newcommand{\sL}{\mathscr L}

\newcommand{\lref}{\labelcref}

\def\todo#1.#2{ %
  \textcolor{Mahogany}{ %
    \footnotesize %
    \newline %
    {\color{Mahogany}\fbox{\parbox{.97\textwidth}{\textbf{#1:} #2}}} %
    \newline %
  } %
}

\definecolor{forrest}{RGB}{81,133,49}
\definecolor{mydarkblue}{RGB}{10,92,153}

\title[Equality in the M.--Y.~inequality and uniformization of klt pairs]{Equality in the Miyaoka--Yau inequality and uniformization of non-positively curved klt pairs}

\dedicatory{In memory of Jean-Pierre Demailly}

\author[Claudon]{Beno\^it Claudon}
\address{Univ Rennes, CNRS, IRMAR --- UMR 6625, F--35000 Rennes, France et Institut Universitaire de France}
\email{\href{mailto:benoit.claudon@univ-rennes1.fr}{benoit.claudon@univ-rennes1.fr}}
\urladdr{\href{https://perso.univ-rennes1.fr/benoit.claudon/}{perso.univ-rennes1.fr/benoit.claudon/}}

\author[Graf]{Patrick Graf}
\address{Lehrstuhl f\"ur Mathematik I, Universit\"at Bayreuth, 95440 Bayreuth, Germany}
\email{\href{mailto:patrick.graf@uni-bayreuth.de}{patrick.graf@uni-bayreuth.de}}
\urladdr{\href{https://patrickgraf.gitlab.io/en/}{www.graficland.uni-bayreuth.de}}

\author[Guenancia]{Henri Guenancia}
\address{Institut de Math\'ematiques de Toulouse, Universit\'e Paul Sabatier, 31062 Toulouse Cedex~9, France}
\email{\href{mailto:henri.guenancia@math.cnrs.fr}{henri.guenancia@math.cnrs.fr}}
\urladdr{\href{https://hguenancia.perso.math.cnrs.fr/}{hguenancia.perso.math.cnrs.fr/}}

\date{May 6, 2023}

\hypersetup{
  pdfauthor={Beno\^it Claudon, Patrick Graf, and Henri Guenancia},
  pdftitle={Equality in the Miyaoka--Yau inequality and uniformization of non-positively curved klt pairs},
  pdfkeywords={Miyaoka--Yau inequality, orbifold uniformization, klt pairs},
  pdfstartview={Fit},
  pdfpagelayout={OneColumn},
  pdfpagemode={UseNone},
  linktoc=all,
  breaklinks,
  linkcolor=rouge2,
  citecolor=violet,
  urlcolor=[RGB]{0 96 0},
  colorlinks}

\begin{document}

\begin{abstract}
Let $(X, \Delta)$ be a compact \kahler klt pair, where $K_X + \Delta$ is ample or numerically trivial, and $\Delta$ has standard coefficients.
We show that if equality holds in the orbifold Miyaoka--Yau inequality for $(X, \Delta)$, then its orbifold universal cover is either the unit ball (ample case) or the affine space (numerically trivial case).
\end{abstract}

\maketitle

\begingroup
\hypersetup{linkcolor=black}
\tableofcontents
\endgroup

\section{Introduction}

Let $X$ be an $n$-dimensional compact \kahler manifold and let us assume that either
\begin{enumerate}[label=(\Roman*)]
\item\label{ample case} $K_X$ is ample (and $X$ is thus projective), or
\item\label{flat case} $K_X$ is numerically trivial (equivalently, $\cc1X = 0$ in $\HH2.X.\R.$).
\end{enumerate}
As a consequence of the existence of a \KE metric $\omega_{\mathrm{KE}}$ on $X$ (proved by Aubin \cite{Aubin} and Yau~\cite{Yau78}), the Chern classes of $X$ satisfy the \emph{Miyaoka--Yau inequality}
\begin{equation} \label{eq:MYsmooth} \tag{MY}
\big( 2(n + 1) \, \cc2X - n \, \cpc12X \big) \cdot \alpha^{n - 2} \ge 0.
\end{equation}
\stepcounter{equation}%
where in case~\lref{ample case}, we set $\alpha = [K_X]$, while in case~\lref{flat case}, $\alpha$ can be an arbitrary \kahler class.
Furthermore, in case of equality, the universal cover $\pi \from \wt X \to X$ is (biholomorphic to)
\begin{enumerate}[label=(\Roman*)]
\setcounter{enumi}{0}
\item the $n$-dimensional unit ball $\B^n = \set{ (z_1, \dots, z_n) \in \C^n \;\;\big|\;\; |z_1|^2 + \cdots + |z_n|^2 < 1 }$,
\item the $n$-dimensional affine space $\C^n$.
\end{enumerate}
We can reformulate the above conclusion by saying that
\begin{enumerate}[label=(\Roman*)]
\setcounter{enumi}{0}
\item $X = \factor{\B^n}\Gamma$ with $\Gamma \subset \mathrm{PU}(1, n) = \Aut(\B^n)$,
\item $X = \factor{\C^n}\Gamma$ with $\Gamma \subset \C^n \rtimes \mathrm{U}(n) = \Aut(\C^n, \pi^* \omega_{\mathrm{KE}})$,
\end{enumerate}
where in both cases, the action of $\Gamma$ on $\wt X$ is \emph{fixed point-free}.

It seems natural to investigate the general case of quotients by cocompact lattices $\Gamma \subset \Aut(\wt X)$ (with $\wt X = \B^n$ or $\C^n$ endowed with the Bergman metric or the flat metric, respectively), the action being of course assumed to be properly discontinuous.
The corresponding quotients are then naturally endowed with an orbifold structure that can be encoded in the datum of a \Q-divisor with standard coefficients (see \cref{set:MY} below).
To sum up, it is natural to consider pairs $(X, \Delta)$ when dealing with these quotients.

The question of uniformizing \emph{spaces} (as opposed to pairs) in the cases~\lref{ample case} and~\lref{flat case} has been considered in the framework of klt singularities.
To quote a few relevant papers:~\cite{GKP16, LuTaji18, GKPT15, GKPT20, GK20, CGG}.
This article grew out of an attempt to understand the general situation with an orbifold structure in codimension one.

Unfortunately, the parallels between cases~\lref{ample case} and~\lref{flat case} cannot be pursued throughout this introductory section since the difficulties (when dealing with the inequality~\lref{eq:MYsmooth} in the singular setting) are not of the same nature.
The following two facts illustrate this point:
\begin{itemize}
\item In case~\lref{ample case}, the variety $X$ is necessarily projective, but the codimension one part of the orbifold structure cannot be easily eliminated.
Therefore we have to use orbifold techniques in the proof.
\item In case~\lref{flat case}, we also need to consider (non-algebraic) compact \kahler spaces, but we can get rid of the codimension one part of the orbifold structure via a cyclic covering (see~\cref{shokurov}).
This enables us to assume that $\Delta = 0$ for most of the argument.
\end{itemize}
Due to this break in symmetry, we split the discussion according to the sign of the canonical bundle.

\subsection*{The canonically polarized case}

Let us recall the singular version of the inequality~\lref{eq:MYsmooth} as proven by the third-named author together with B.~Taji~\cite{GT16}.
When dealing with case~\lref{ample case}, we work in the following setting:

\begin{setup} \label{set:MY}
Let $(X, \Delta)$ be an $n$-dimensional klt pair, where $X$ is a projective variety and $\Delta$ has standard coefficients, i.e.~$\Delta = \sum_{i \in I} \big( 1 - \frac1{m_i} \big) \Delta_i$ with integers $m_i \ge 2$ and the $\Delta_i$ irreducible and pairwise distinct.
\end{setup}

\begin{thm}[\protect{$\subset$~\cite[Thm.~B]{GT16}}] \label{th:MY ineq}
Let $(X, \Delta)$ be as in \cref{set:MY}, and assume that $K_X + \Delta$ is big and nef.
Assume additionally that every irreducible component $\Delta_i$ of $\Delta$ is \Q-Cartier.
Then the following inequality holds:
\begin{equation} \label{eq:MY}
\big( 2(n + 1) \, \ccorb2{X, \Delta} - n \, \cpcorb12{X, \Delta} \big) \cdot [K_X + \Delta]^{n - 2} \ge 0.
\end{equation}
Here, $\ccorb2{X, \Delta}$ and $\cpcorb12{X, \Delta}$ denote the appropriate orbifold Chern classes of the pair $(X, \Delta)$, as defined e.g.~in~\cite[Notation~3.7]{GT16}. \qed
\end{thm}

\begin{rem-plain}
In the above theorem, the assumption that the $\Delta_i$ be \Q-Cartier is not necessary, and establishing this is one of the (minor) contributions of this paper, cf.~\cref{GT16 B}.
While this may seem like an innocuous technical issue at first sight, eliminating the \Q-Cartier assumption will become crucial below when deducing \cref{T:general type} from \cref{MY eq}, see \cref{qcartier}.
\end{rem-plain}

As in the smooth case, it is interesting to characterize geometrically those pairs that achieve equality in~\lref{eq:MY}.
In the case where $\Delta = 0$, this has been achieved in~\cite[Thm.~1.2]{GKPT15} and~\cite[Thm.~1.5]{GKPT20}:
equality holds if and only if there is a finite \qe Galois cover $Y \to X$ such that the universal cover of $Y$ is the unit ball.
An expectation concerning the general case was formulated in~\cite[Sec.~10.2]{GKPT15}.
Our first main result confirms this expectation.

\begin{bigthm}[Uniformization of canonical models] \label{MY eq}
Let $(X, \Delta)$ be as in \cref{set:MY}.
Assume that $K_X + \Delta$ is ample and that equality holds in~\lref{eq:MY}.
Then the orbifold universal cover $\pi \from \wt X_\Delta \to X$ of $(X, \Delta)$ is the unit ball (cf.~\cref{univ cov}).
More precisely, $(\wt X_\Delta, \wt \Delta) \isom (\B^n, \emptyset)$.
\end{bigthm}

\noindent
In fact, a suitable converse of the above theorem also holds, and we obtain the following corollary.

\begin{cor}[Characterization of ball quotients] \label{main3}
Let $(X, \Delta)$ be as in \cref{set:MY}.
The following are equivalent:
\begin{enumerate}
\item\label{main3.1} $K_X + \Delta$ is ample, and equality holds in~\lref{eq:MY}.
\item\label{main3.2} The orbifold universal cover of $(X, \Delta)$ is the unit ball $\B^n$.
\item\label{main3.3} $(X, \Delta)$ admits a finite orbi-\'etale Galois cover $f \from Y \to X$ (cf.~\cref{adapted}), where $Y$ is a projective manifold whose universal cover is the unit ball.
\end{enumerate}
\end{cor}

\noindent
In the spirit of~\cite[Thm.~1.5]{GKPT20}, we can also prove the following uniformization statement for minimal pairs of log general type.

\begin{cor}[Uniformization of minimal models] \label{T:general type}
Let $(X, \Delta)$ be as in \cref{set:MY}.
Assume that $K_X + \Delta$ is big and nef and that equality holds in~\lref{eq:MY}.
Then the canonical model $(X, \Delta)_\canmod \ndef (X_\canmod, \Delta_\canmod)$ of the pair $(X, \Delta)$ is a ball quotient in the sense of \cref{MY eq}.
\end{cor}

\subsection*{The flat case}

As mentioned earlier, \kahler quotients of $\C^n$ by cocompact groups of isometries are in general not projective, so we have to consider the following framework.

\begin{setup} \label{setup:flat}
Let $(X, \Delta)$ be an $n$-dimensional klt pair, where $X$ is a compact \kahler space and $\Delta$ has standard coefficients, i.e.~$\Delta = \sum_{i \in I} \big( 1 - \frac1{m_i} \big) \Delta_i$ with integers $m_i \ge 2$ and the $\Delta_i$ irreducible and pairwise distinct.
\end{setup}

In this more general \kahler setting, the methods of~\cite{GT16} cannot be used to prove a singular analogue of the Miyaoka--Yau inequality.
Instead, we rely on the Decomposition Theorem from~\cite{BakkerGuenanciaLehn20} to deduce the following singular version of the inequality~\lref{eq:MYsmooth} in case~\lref{flat case}.

\begin{thm}[Singular Miyaoka--Yau inequality] \label{th:MY singular Kahler}
Let $(X, \Delta)$ be as in~\cref{setup:flat} and assume that $\cc1{K_X + \Delta} = 0 \in \HH2.X.\R.$.
Let $\alpha \in \HH2.X.\R.$ be any \kahler class.
We then have:
\begin{equation} \label{ineq MY Kahler}
\ccorb2{X, \Delta} \cdot \alpha^{n-2} \geq 0.
\end{equation}
\end{thm}

\noindent
As before, we are particularly interested in what happens if equality is achieved.

\begin{bigthm}[Uniformization in the flat case] \label{th:torus quotients}
Let $(X, \Delta)$ be as in \cref{setup:flat}.
Assume that $\cc1{K_X + \Delta} = 0 \in \HH2.X.\R.$ and that equality holds in~\lref{ineq MY Kahler} for \emph{some} \kahler class~$\alpha$.
Then the orbifold universal cover $\pi \from \wt X_\Delta \to X$ of $(X, \Delta)$ is the affine space (cf.~\cref{univ cov}).
More precisely, $(\wt X_\Delta, \wt \Delta) \isom (\C^n, \emptyset)$.
\end{bigthm}

\noindent
As above, we can formulate a converse and get the following corollary.

\begin{cor}[Characterization of torus quotients] \label{th:characterization torus quotient}
Let $(X, \Delta)$ be as in \cref{setup:flat}.
The following are equivalent:
\begin{enumerate}
\item\label{torus.1} $\cc1{K_X + \Delta} = 0 \in \HH2.X.\R.$, and equality holds in~\lref{ineq MY Kahler} for some \kahler class~$\alpha$.
\item\label{torus.2} The orbifold universal cover of $(X, \Delta)$ is $\C^n$.
\item\label{torus.3} $(X, \Delta)$ admits a finite orbi-\'etale Galois cover $f \from T \to X$ (cf.~\cref{adapted}), where $T$ is a complex torus.
\end{enumerate}
\end{cor}

\noindent
The previous statements are thus generalizations of~\cite[Thm.~1.2]{LuTaji18} (itself elaborating on~\cite[Thm.~1.17]{GKP16}).
The generalization is threefold:
\begin{itemize}
\item Here $X$ is a compact \kahler space, not necessarily projective.
\item The class $\alpha$ is transcendental, a priori not an ample class.
\item Ramification is allowed in codimension one; i.e.~we work with klt pairs rather than klt spaces.
\end{itemize}

\subsection*{Acknowledgements}
We are honored to dedicate this paper to the memory of Jean-Pierre Demailly, who has been a constant source of inspiration and admiration to us.

B.C.~would like to thank Institut Universitaire de France for providing excellent working conditions.
H.G.~acknowledges the support of the French Agence Nationale de la Recherche (ANR) under reference ANR-21-CE40-0010.

\section{Generalities on orbifolds} \label{sec:orbifolds}

In this section, we consider Kawamata log terminal (klt) pairs $(X, \Delta)$ consisting of a normal algebraic variety or complex space $X$ of dimension $n$ and a \Q-divisor $\Delta = \sum_{i \in I} \big( 1 - \frac 1 {m_i} \big) \Delta_i$ on $X$, with $m_i \ge 2$.

\subsection{Orbi-structures and orbi-sheaves}

Most of the definitions and basic properties given below can be found in e.g.~\cite[\S 2]{GT16} in the slightly more general setting of \emph{dlt} pairs with standard coefficients, at least if $X$ is algebraic.
Working exclusively with klt pairs will simplify the exposition.

\begin{defi}[Adapted morphisms] \label{adapted}
Let $f \from Y \to X$ be a finite surjective Galois morphism from a normal variety or complex space $Y$.
One says that $f$ is:
\begin{itemize}
\item \emph{adapted} to $(X, \Delta)$ if for all $i \in I$, there exists $a_i \in \Z^{\ge 1}$ and a reduced divisor $\Delta_i'$ on $Y$ such that $f^* \Delta_i = a_i m_i \Delta_i'$,
\item \emph{strictly adapted} to $(X, \Delta)$ if it is adapted and if $a_i = 1$ for all $i \in I$,
\item \emph{orbi-étale} if it is strictly adapted and the divisorial component of the branch locus of $f$ is contained in $\supp(\Delta)$ --- equivalently, if $f$ is \'etale over $\reg X \setminus \supp(\Delta)$.
\end{itemize}
\end{defi}

\begin{rem-plain}
If $X$ is compact, then a map $f \from Y \to X$ as above is orbi-\'etale if and only if $K_Y = f^*(K_X + \Delta)$.
\end{rem-plain}

\begin{defi}[Orbi-structures] \label{orbi structures}
An \emph{orbi-structure} for the pair $(X, \Delta)$ consists of a compatible collection of triples $\cC = \big\{ ( U_\alpha, f_\alpha, X_\alpha ) \big\}_{\alpha \in J}$, where $(U_\alpha)_{\alpha \in J}$ is a covering of $X$ by étale-open subsets, and for each $\alpha \in J$, $f_\alpha \from X_\alpha \to U_\alpha$ is an adapted morphism from a normal complex space $X_\alpha$ with respect to the pair structure on $U_\alpha$ induced by $(X, \Delta)$.
The compatibility condition means that for all $\alpha, \beta \in J$, the projection map $g_{\alpha \beta} \from X_{\alpha \beta}\to X_\alpha$ is quasi-étale, where $X_{\alpha \beta}$ is the normalization of $X_\alpha \times_X X_\beta$.

An orbi-structure $\cC = \big\{ ( U_\alpha, f_\alpha, X_\alpha ) \big\}_{\alpha \in J}$ is called \emph{strict} (resp.~\emph{orbi-étale}) if for each $\alpha \in J$, the morphism $f_\alpha$ is strictly adapted (resp.~orbi-étale).
It is called \emph{smooth} if for each $\alpha \in J$, the variety $X_\alpha$ is smooth.
In this case, the maps $g_{\alpha \beta}$ are étale by purity of branch locus.
\end{defi}

\begin{defi}[Quotient singularities] \label{quot sing}
A pair $(X, \Delta)$ is said to have \emph{quotient singularities} if locally analytically on $X$, there exists an orbi-\'etale morphism $f \from Y \to X$, where $Y$ is smooth.
The maximal open subset of $X$ where this condition is satisfied will also be referred to as the \emph{orbifold locus} of $(X, \Delta)$ and will be denoted by $X^\circ \subset X$ or $X^\orb \subset X$.
\end{defi}

\begin{rem-plain}
With the above terminology, a pair $(X, \Delta)$ admits a smooth orbi-\'etale orbi-structure if and only if it has quotient singularities.
This is because the compatibility condition is automatically satisfied.
\end{rem-plain}

The following technical result will be useful in the sequel: a pair with quotient singularities whose underlying space is compact \kahler is a \kahler orbifold.
The log smooth case had been already observed in~\cite[Prop.~2.1]{Claudon08}.
Slightly more generally, we have the following.

\begin{lem}[Existence of orbifold \kahler metrics] \label{lem:kahler orbi}
Let $(Z, \Delta)$ be a pair with quotient singularities and such that $Z$ is a \kahler space.
Then for any relatively compact open subset $X \Subset Z$, there exists an orbifold \kahler metric $\omega$ adapted to $(X, \Delta|_{X})$ in the sense that $\omega$ is a \kahler metric on $\reg X \setminus \supp \Delta$ which pulls back to a smooth \kahler metric on the smooth local covers.
\end{lem}

\begin{proof}
One can find an open neighborhood $X'$ of $\overline X \subset Z$ admitting a finite covering $X'=\bigcup_{\alpha\in I} X'_\alpha$ such that there exist smooth orbi-étale covers $p_\alpha\from Y'_\alpha\to X'_{\alpha}$. We set $X_\alpha \defn X'_\alpha \cap X$ and $Y_\alpha \defn p_\alpha^{-1}(X_\alpha)$.
We pick a Kähler metric $\om_Z$ on $Z$, as well as potentials $\phi_\alpha$ on $X_\alpha'$ such that  $dd^c p_\alpha^*\phi_\alpha$ is a Kähler metric on $Y_\alpha'$;  the functions $\phi_\alpha$ are solely continuous on $X_\alpha$ but $p_\alpha^*\phi_\alpha$ is smooth on $Y_\alpha'$.  We can assume that $|\phi_\alpha| \le 1$ on $X_\alpha$. Finally, let $(\chi_\alpha)_{\alpha \in I}$ be some partition of unity subordinate to the covering $(X_\alpha)_{\alpha\in I}$ and set $\phi \defn \sum \chi_\alpha \phi_\alpha$. We set $N \defn |I|$ and pick a constant $C>0$ such that 
\begin{equation}
\label{eq C}
\|dd^c\chi_\alpha\|_{\om_Z}^2+\|d\chi_\alpha\|^2_{\om_Z} \le C,
\end{equation} 
holds for any $\alpha\in I$ and we claim that the current
\[\omega \defn M\om_Z+dd^c \phi\]
is an orbifold Kähler metric on $X$ for $M\gg 1$. Clearly, $\omega$ is smooth as an orbifold differential form, as one can see directly by using the compatibility of the covers. Let $x\in X$ and let $J \defn \{\alpha\in I, x\in X_\alpha\}=\{\alpha_1, \ldots, \alpha_s\}$. We set $X_J \defn \cap_{\alpha \in J} X_\alpha$ and choose a connected component $Y_J$ of the normalization of $p_{\alpha_1}^{-1}(X_J)\times_{X_J} \cdots \times_{X_J} p_{\alpha_s}^{-1}(X_J)$.
The space $Y_J$ is a smooth manifold endowed with an orbi-étale map $p_J \from Y_J \to X_J$ induced by the $p_{\alpha_i}$, $i = 1, \dots, s$.

We have $1=\sum_{\alpha \in I} \chi_\alpha(x)=\sum_{\alpha \in J} \chi_\alpha(x)$, hence there exists $\beta\in J$ such that $\chi_\beta(x) \ge \frac 1N$. 
Since $p_J^*(dd^c\phi_\beta|_{X_J})$ is a Kähler metric on $Y_J$ (which extends slightly beyond), we infer that there exists $\delta>0$ such that 
\[\forall \alpha \in J, \quad dd^c\phi_\beta \ge \delta d\phi_\alpha \wedge d^c\phi_\alpha \quad \mbox{on } \, X_J.\]
Next, we have the following inequality for any $\ep>0$:
\[\pm   (d\phi_\alpha \wedge d^c\chi_\alpha+d\chi_\alpha \wedge d^c\phi_\alpha) \le \ep d\phi_\alpha \wedge d^c\phi_\alpha+\ep^{-1} d\chi_\alpha \wedge d^c\chi_\alpha.\]
Combining the above inequality with~\lref{eq C}, we get for any $\ep>0$:
\begin{align*}
\omega&=M\om_Z+\sum_{\alpha\in I} \chi_\alpha dd^c \phi_\alpha+\sum_{\alpha\in I} \phi_\alpha dd^c  \chi_\alpha+ \sum_{\alpha\in I}(d\phi_\alpha \wedge d^c \chi_\alpha+d\chi_\alpha \wedge d^c \phi_\alpha)\\
&\ge (M-NC(1+\ep^{-1}))\om_Z+ \chi_\beta dd^c \phi_\beta-\ep \sum_{\alpha\in I}d\phi_\alpha \wedge d^c \phi_\alpha
\end{align*}
which yields, at the point $x$: 
\[\omega\ge (M-NC(1+\ep^{-1}))\om_Z+\left(\frac 1N -\frac{N\ep}{\delta}\right) dd^c \phi_\beta.\]
Therefore, if we choose $\ep \defn \frac{\delta}{2N^2}$ and $M=2 NC(1+\ep^{-1})$, then $\om$ is an orbifold Kähler metric near $x$. Since $x$ is arbitrary and the constants $N,C,\delta$ are uniform, the lemma is now proved. 
\end{proof}

\subsection{Covering constructions}

In what follows, we present some variations on the well-known cyclic covering theme.
The first one, \cref{shokurov}, is a consequence of~\cite[Ex.~2.4.1]{Shokurov92} when $X$ is quasi-projective so that $K_X$ is well-defined as a (class of) Weil divisor, but one needs to argue slightly differently in the complex analytic case.
The second one, \cref{ex strictly adapted}, improves upon previous results such as~\cite[Prop.~2.9]{JabbuschKebekus11}, \cite[Ex.~2.11]{GT16} and~\cite[Prop.~2.38]{ClaudonKebekusTaji21}.
The main observation is that given a pair $(X, \Delta)$, it is (for our purposes) unnecessary to assume that the components of $\Delta$ are \Q-Cartier as long as $K_X + \Delta$ is.
As explained in \cref{qcartier}, this is crucial for proving \cref{T:general type}.

\begin{prop}[Existence of orbi-\'etale covers] \label{shokurov}
Let $(X, \Delta)$ be a (not necessarily klt) pair with standard coefficients, where $X$ is a normal complex space.
Assume that there is a reflexive rank~$1$ sheaf $\sL$ and an integer $N \ge 1$ such that $N \Delta$ is a \Z-divisor and
\[ \O X( N \Delta ) \isom \sL^{[N]}. \]
Then there exists an orbi-\'etale morphism $f \from Y \to X$. In particular:

If $(X, \Delta)$ is klt and there is an integer $N \ge 1$ such that $N \Delta$ is a \Z-divisor and $\can X^{[N]}(N \Delta) \isom \O X$, then we can find an orbi-\'etale morphism $f \from Y \to X$ such that $\can Y \isom \O Y$ and $Y$ has canonical singularities.
\end{prop}

\begin{proof}
Let $\sigma \in \HH0.X.\sL^{[N]}.$ be such that $\operatorname{div}(\sigma) = N \Delta$, and let us consider the cyclic covering $g \from Z \to X$ induced by $\sigma$, cf.~e.g.~\cite[Def.~2.52]{KM}.
In the analytic setting, we can construct $f$ in the following way.
On $\reg X\setminus \supp(\Delta)$, $\sL |_{\reg X\setminus \supp(\Delta)}$ is torsion and it gives rise to an \'etale cover $g^\circ\from Z^\circ\to\reg X\setminus \supp(\Delta)$ (the $N^{\text{th}}$-root of $\sigma |_{ \reg X\setminus \supp(\Delta)}$) that is moreover a Galois cover with cyclic Galois group.
According to~\cite[Thm.~3.4]{DethloffGrauert94}, the map $g^\circ$ can be extended to a finite cover $f \from Z \to X$ with the same Galois group.

We claim that $g$ ramifies exactly at order $m_i$ along $\Delta_i$. It is enough to check the claim at a general point of $\Delta_i$. Therefore, there is no loss of generality assuming that  $(X,\Delta)=(U,(1-\frac 1m) D)$ where $U\subset \mathbb C^n$ ($n=\dim(X)$) is a ball, $D=(z_1=0)\cap U$, and that $\sigma|_U= z_1^{N(1-\frac 1m)} \sigma_{\sL,\, U}^{\otimes N}$ with $\sigma_{\sL,\, U}$ a trivializing section of $\sL$ over $U$.

Write $N=km$, and let $V \defn \{(t,z)\in \C \times \C^n \, \big| \, t^N=z_1^{k(m-1)}\} \subset \C\times \C^n$ and let $\nu\from V^\nu\to V$ be its normalization. One can actually write down exactly what $V^\nu$ is. Indeed, let $\zeta$ be a primitive $k$-th root of unity, and set $V_p \defn \{(t,z)\, \big| \, t^m=\zeta^p z_1^{m-1}\} \subset \C\times \C^n$ for $p=0,\ldots, k-1$. We have a decomposition $V=\cup_p V_p$ into irreducible components, and the normalization $\nu_p\from V_p^\nu \to V_p$ is the affine space $V_p^\nu \isom \C\times \C^{n-1}$ with map $\nu_p(u,w)=(\xi u^{m-1},u^m, w)$ where $\xi$ is an $m$-th root of $\zeta^p$. Now, set $V^\nu \defn \sqcup_p V_p^\nu$ and defined $\nu\from V^{\nu}\to V$ by $\nu|_{V_p^\nu} \defn \nu_p$.
We have a diagram
\[
\begin{tikzcd}
V^\nu \arrow[rr, bend left=35, dotted, "j"] \arrow[r,"\nu"', swap]&   
 V \arrow[d, "\mathrm{pr}_{\mathbb C^n}", swap]  & Z \arrow[d, "g"]\\
&  U \arrow[r, hookrightarrow] & X
\end{tikzcd}
\]
where $j$ is obtained by the universal property of normalization.
In particular, $j$ is finite and generically $1$-to-$1$ between normal varieties, hence it is an open embedding.
Moreover, if $(u,w)\in V_p^\nu$, we have $\mathrm{pr}_{\mathbb C^n} \circ \nu (u,w)=(u^m,w)$, hence the latter map ramifies at order $m$ along $D$.
It follows that $g$ ramifies at order $m$ along $D$.

Finally, one picks one irreducible component $Y$ of $Z$ and sets $f \defn g|_Y$.
It yields the expected cover, which is Galois with group $G < \Z/n\Z \isom \Gal( Z \to X )$ defined as the stabilizer of $Y$.

As for the last part of the proposition, we can apply the above construction to $\sL = \can X^{[-1]} \defn \can X \dual$.
This provides us with an orbi-\'etale morphism $f \from Y \to X$.
In particular, $Y$~is klt and the computations made above show that $f^* ( K_X + \Delta )$ is trivial over $\reg X \setminus \sing \Delta$.
So we get that $\can Y$ is trivial as well and finally that $Y$ has only canonical singularities.
\end{proof}

\begin{prop}[Existence of strictly adapted covers] \label{ex strictly adapted}
Let $(X, \Delta)$ be a projective pair with standard coefficients such that $K_X + \Delta$ is \Q-Cartier (but not necessarily klt).
Then there exists a very ample divisor $L$ on $X$ such that for general $H \in |L|$, there exists a cyclic Galois cover $f \from Y \to X$ with the following properties:
\begin{enumerate}
\item\label{esa.0} The morphism $f$ is orbi-\'etale for $\big( X, \Delta + (1 - \frac1N) H \big)$, where $N \defn \deg(f)$.
\item\label{esa.1} The morphism $f$ is strictly adapted for $(X, \Delta)$.
\item\label{esa.2} If $(X, \Delta)$ is klt, then so are the pairs $\big( X, \Delta + (1 - \frac1N) H \big)$ and $(Y, \, \emptyset)$.
\end{enumerate}
\end{prop}

\begin{proof}
Pick, once and for all, a representative $K$ of $K_X$, that is, an integral (but not necessarily effective) Weil divisor $K$ on $X$ such that $K_X \sim K$.
Choose a very ample divisor $A$ on $X$ and a positive integer $N$ such that
\[ L \defn N \cdot \big( A - ( K + \Delta ) \big) \]
is integral and very ample, and pick a general element $H \in |L|$.
Consider the principal divisor
\[ D \defn H - L = H + N \cdot ( K + \Delta - A ) \sim 0. \]
Let $f \from Y \to X$ be the degree $N$ cyclic cover associated to $D$, as in~\cite[\S 2.3]{Shokurov92}.
(To be more precise, $Y$ is an arbitrary irreducible component of the normalization of that cover.)
We need to check properties~\lref{esa.0}--\lref{esa.2}.

By construction, the branch locus of $f$ is contained in $\supp(D)$.
Recall from~\cite{Shokurov92} that writing $D = \sum_i d_i D_i$, the ramification order of $f$ along each component of $f\inv(D_i)$ is given by $N / \operatorname{hcf}(d_i, N)$.
Since $K$, $A$ and $H$ are \Z-divisors, where $H$ is even reduced, this implies~\lref{esa.0}.
Property~\lref{esa.1} is an immediate consequence.

For~\lref{esa.2}, it is enough to show the first claim thanks to~\lref{esa.0} and~\cite[Prop.~5.20]{KM}.
To check the claim, we take a log resolution $\pi \from \wt X \to X$ of $(X, \Delta)$ and write
\[ K_{\wt X} + \Delta' = \pi^*( K_X + \Delta ) + \sum a_i E_i \]
as usual, where $\Delta'$ is the strict transform of $\Delta$.
Since $H$ is a general element of $|L|$, and $\pi^* |L|$ is basepoint-free, one can assume that $\pi^* H = \pi^{-1}_* H$ is smooth and intersects each stratum of the exceptional divisor of $\pi$ and of $\Delta'$ smoothly.
In particular, $\pi$ is also a log resolution for the pair $\big( X, \Delta + ( 1 - \frac 1 N ) H \big)$.
Now, the identity
\[ K_{\wt X} + \Delta' + \left( 1 - \frac 1 N \right) \pi^{-1}_* H = \pi^* \left( K_X + \Delta + \left( 1 - \frac 1 N \right) H \right) + \sum a_i E_i \]
shows that $\big( X, \Delta + ( 1 - \frac 1 N ) H \big)$ is klt.
\end{proof}

\begin{rem-plain}
More generally, it can be observed that a pair $(X, \Delta)$ (with $X$ a normal analytic space) admits strictly adapted covers if there exists a Cartier divisor $D$ on $X$ having no component in common with $\Delta$ and such that $m(K_X + \Delta) \sim D$ for some (sufficiently divisibe) integer $m \ge 1$.
We can indeed apply \cref{shokurov} to the pair $(X \setminus D, \Delta|_{X \setminus D})$ and get an orbi-\'etale cover $Y^\circ\to X \setminus D$.
Its completion over $X$ is then adapted with respect to $\Delta$ and the extra-ramification is supported over the components of $D$.
\end{rem-plain}

The following result seems to have been known to experts for a long time.
A proof of it was written down in~\cite{GKKP} in the case where $\Delta = 0$, and the general case follows almost immediately from \cref{shokurov} as we will explain.

\begin{lem}[Klt pairs have quotient singularities in codimension two] \label{klt orb}
Let $(X, \Delta)$ be a klt pair with standard coefficients.
Then there is a Zariski closed subset $Z \subset \sing X \cup \supp \Delta$ with $\codim XZ \ge 3$ such that for $X^\circ \defn X \setminus Z$, the pair $(X^\circ, \Delta|_{X^\circ})$ admits a smooth orbi-\'etale orbi-structure $\cC^\circ$.
\end{lem}

\begin{proof}
Since $K_X + \Delta$ is a \Q-Cartier divisor, we can cover $X$ by (affine or Stein) open subsets $U_\beta \subset X$, $\beta \in I$, such that $( K_X + \Delta )|_{U_\beta} \sim_\Q 0$.
By~\cref{shokurov}, we can find a finite cyclic cover $g_\beta \from U_\beta' \to U_\beta$ that branches exactly over the $\Delta_i|_{U_\beta}$ with multiplicity $m_i$. 
Moreover, $U_\beta'$ has klt singularities, since $K_{U_\beta'} = g_\beta^* ( K_{U_\beta} + \Delta|_{U_\beta} )$.
We can now use~\cite[Prop.~9.3]{GKKP} or~\cite[Lemma~5.8]{GK20} to find a \emph{smooth} orbi-\'etale orbi-structure $\{U'_\bg, f_\bg, X'_\bg\}_{\gamma \in J}$ on $U_\beta' \setminus Z_\beta$, for some closed subset $Z_\beta \subset U_\beta'$ of codimension at least three.
Set $U_\bg = g_\beta(U_\bg')$, so that $\bigcup_\beta U_{\bg} \subset U_\beta$ is an open subset whose complement is of codimension at least three. In summary, we get the following diagram:
\begin{equation} \label{cover}
\begin{tikzcd}
X'_\bg \arrow[r, "f_\bg"] \arrow[rr, "h_\bg", bend left]&   U_\bg' \arrow[r, "g_\beta"]  \arrow[d, hookrightarrow] &U_\bg \arrow[d, hookrightarrow] \\
& U_\beta' \arrow[r, "g_\beta"] & U_\beta \arrow[r, hookrightarrow] & X
\end{tikzcd}
\end{equation}
Now $\set{ U_\bg, h_\bg, X'_\bg }_{(\beta, \gamma) \in I \times J}$ is the sought-after smooth orbi-\'etale orbi-structure on $(X^\circ, \Delta|_{X^\circ})$, where the open subset $X^\circ \defn \bigcup_{(\beta, \gamma) \in I \times J} U_{\beta\gamma}$ has complement of codimension at least three.
\end{proof}

\begin{rem} \label{quot sing surf}
In particular, a klt surface pair with standard coefficients admits a smooth orbi-\'etale orbi-structure, hence it has quotient singularities in the sense of~\cref{quot sing}.
This is of course well-known and follows from the cyclic cover construction recalled above and~\cite[Proposition~4.18]{KM}.
\end{rem}

\begin{defi}[Orbi-sheaves]
An \emph{orbi-sheaf} with respect to an orbi-structure $\cC = \big\{ ( U_\alpha, f_\alpha, X_\alpha ) \big\}_{\alpha \in J}$ on $(X, \Delta)$ is the datum of a collection $(\sE_\alpha)_{\alpha \in J}$ of coherent sheaves on each $X_\alpha$, together with isomorphisms $g_{\alpha \beta}^* \sE_\alpha \isom g_{\beta \alpha}^* \sE_\beta$ of $\O{X_{\alpha \beta}}$-modules satisfying the natural compatibility conditions on triple overlaps.
\end{defi}

All the usual notions for sheaves (locally free, reflexive, subsheaves, morphisms etc.) can be carried over to this setting in the obvious way, cf.~\cite[\S 2.7]{GT16}.
Ditto for Higgs fields and Higgs sheaves, cf.~\cite[Definition~2.24]{GT16}.

Recall the following definition from~\cite[Sec.~3]{ClaudonKebekusTaji21}:

\begin{defi}[Adapted differentials] \label{adapt diff}
Let $\gamma \from Y \to X$ be a strictly adapted morphism for $(X, \Delta)$.
Let $X^\circ \subset X$ and $\iota \from Y^\circ \inj Y$ be the maximal open subsets where $\gamma$ is \emph{good} in the sense of~\cite[Def.~3.5]{ClaudonKebekusTaji21}.
The sheaf of \emph{adapted reflexive differentials} is defined as
\[ \Omegar{(X, \Delta, \gamma)}1 \defn \iota_* \bigg[ \Big( \img \!\big( \gamma^* \Omegap{X^\circ}1 \to \Omegap{Y^\circ}1 \big) \tensor \O{Y^\circ}(\gamma^* \Delta) \Big) \cap \Omegap{Y^\circ}1 \bigg]. \]
\end{defi}

\begin{lem} \label{adapt diff lem}
The following properties hold:
\begin{enumerate}
\item\label{adl.1} The sheaf $\Omegar{(X, \Delta, \gamma)}1$ is a coherent reflexive subsheaf of $\Omegar Y1$.
\item\label{adl.2} If $\gamma$ is orbi-\'etale for $(X, \Delta)$, then $\Omegar{(X, \Delta, \gamma)}1 = \Omegar Y1$.
\item\label{adl.3} Let $\gamma_2 \from Z \to Y$ be quasi-\'etale, where $Z$ is normal.
Then $\delta \defn \gamma \circ \gamma_2 \from Z \to X$ is strictly adapted for $(X, \Delta)$, and $\Omegar{(X, \Delta, \delta)}1 = \gamma_2^{[*]} \Omegar{(X, \Delta, \gamma)}1$. \qed
\end{enumerate}
\end{lem}

\begin{defi}[Orbifold cotangent sheaf, cf.~\protect{\cite[Def.~2.23]{GT16}}] \label{orbi cotang}
Consider on $(X, \Delta)$ any strictly adapted orbi-structure $\cC = \big\{ ( U_\alpha, f_\alpha, X_\alpha ) \big\}_{\alpha \in J}$.
Then the sheaves \[\big( \Omegar{(X, \Delta, f_\alpha)}1 \big)_{\alpha \in J}\] induce a reflexive orbi-sheaf called the \emph{orbifold cotangent sheaf}, or \emph{sheaf of reflexive differential forms}, which we denote by $\Omegar\cC1$.
If the orbi-structure $\cC$ is smooth and orbi-\'etale, then $\Omegar\cC1$ is locally free.
Changing the (strictly adapted) orbifold structure yields compatible sheaves in the sense of~\cite[Def.~3.2]{GT16}, hence we will often denote this sheaf by $\Omegar{(X, \Delta)}1$.

The same construction can be carried out for any integer $p \ge 0$, yielding orbi-sheaves~$\Omegar{(X, \Delta)}p$.
For $p = 0$, we obtain the \emph{structure sheaf} $\O{(X, \Delta)}$, which is nothing but $\O{X_\alpha}$ in each chart $f_\alpha$.
\end{defi}

\begin{lem} \label{surface}
Let $(X, \Delta)$ be a projective klt pair with standard coefficients, and let $X^\circ$ be endowed with a smooth orbi-étale orbi-structure $\cC$ as in \cref{klt orb}.
Let $H$ be an ample line bundle on $X$ and pick a complete intersection surface
\[ S = D_1 \cap \cdots \cap D_{n - 2} \]
of $n - 2$ general hypersurfaces $D_i \in |mH|$ for $m \gg 1$.
Then $S \subset X^\circ$ and the restriction of $\cC$ to $(S, \Delta|_S)$ induces a smooth orbi-étale orbi-structure on $(S, \Delta|_S)$.
In particular, $(S, \Delta|_S)$ has quotient singularities.
\end{lem}

\begin{proof}
We have $S\subset X^\circ$ for dimensional and genericity reasons.
Next, if we express the structure $\cC$ as $\cC = \big\{ ( X_\alpha, f_\alpha, U_\alpha ) \big\}$, set $S_\alpha \defn S \cap U_\alpha$, $T_\alpha \defn f_\alpha^{-1}(S_\alpha)$, $g_\alpha \defn f_\alpha|_{T_\alpha}$, and define $\cC|_S \defn \big\{ ( T_\alpha, g_\alpha, S_\alpha ) \big\}$.
We claim that $T_\alpha$ is smooth, which would prove the lemma.
Indeed, since $f_\alpha$ is quasi-finite (as the composition of an étale map with a finite map), one can find an open immersion $X_\alpha \inj \overline{X_\alpha}$ and a finite extension $\overline{f_\alpha} \from \overline{X_\alpha} \to X$ of $f_\alpha$ as follows:
\[
\begin{tikzcd}
T_\alpha \arrow[r, hookrightarrow] \arrow[d, "g_\alpha"] &   X_\alpha \arrow[d,"f_\alpha"]  \arrow[r,hookrightarrow] & \overline{X_\alpha} \arrow[d,"\overline{f_\alpha}"] \\
 S_\alpha   \arrow[r, hookrightarrow]& U_\alpha \arrow[r,hookrightarrow] & X
\end{tikzcd}
\]
Since $\overline{f_\alpha}^* |mH|$ is basepoint-free, Bertini's theorem guarantees that if $\overline{T_\alpha}$ is a general intersection of $(n - 2)$ hypersurfaces in $\overline{f_\alpha}^* |mH|$, then $\overline{T_\alpha} \cap \overline{X_\alpha}^{\mathrm{reg}}$ is smooth.
Since $X_\alpha \subset \overline{X_\alpha}^{\mathrm{reg}}$, this shows that $T_\alpha$ is smooth, hence the lemma.
\end{proof}

\subsection{The orbifold fundamental group}

Let $(X, \Delta)$ be a klt pair with standard coefficients as before, and set $X^* \defn \reg X \setminus \supp \Delta$.

\begin{defi}[Fundamental group] \label{def pi1orb}
The \emph{(orbifold) fundamental group} of $(X, \Delta)$ is defined as
\[ \piorb{X, \Delta} \defn \factor{\pi_1(X^*)}{\la\!\la \gamma_i^{m_i}, \, i \in I \, \ra\!\ra}. \]
Here, for each $i \in I$, the element $\gamma_i$ is a ``loop around $\Delta_i$'', i.e.~a loop in the normal circle bundle of $\reg{(\Delta_i)} \cap \reg X \subset \reg X$, and $\la\!\la \cdots \ra\!\ra$ denotes the normal subgroup generated by a given subset.
\end{defi}

Note that if $D = \emptyset$, then $\piorb{X, \emptyset} = \pi_1(\reg X)$ is in general different from~$\pi_1(X)$.

\begin{defi}[Covers branched at $\Delta$, cf.~\protect{\cite[Def.~1.3]{Claudon08}}] \label{orbifold cover}
A cover of $X$ \emph{branched at most at $\Delta$} is a holomorphic map $\pi \from Y \to X$, where:
\begin{enumerate}
\item $Y$ is a normal connected complex space (not necessarily quasi-projective),
\item $\pi$ has discrete fibres and $\pi\inv(X^*) \to X^*$ is \'etale,
\item\label{709} at each irreducible component $\wt\Delta_{j, k} \subset \pi\inv(\Delta_j)$, the ramification index $r_{j, k}$ of $\pi$ divides $m_j$,
\item\label{743} every $x \in X$ has a connected neighborhood $V \subset X$ such that every connected component $U$ of $\pi\inv(V)$ meets the fibre $\pi\inv(x)$ in only one point, and $\pi|_U \from U \to V$ is finite.
\end{enumerate}
We say that $\pi$ is \emph{branched exactly at $\Delta$} if in~\lref{709}, we have $r_{j, k} = m_j$ for all $j, k$.
\end{defi}

Note that if $Y$ is quasi-projective and $\pi$ is Galois, then saying that $\pi$ is branched exactly at $\Delta$ is the same as saying that $\pi$ is orbi-\'etale.

\begin{thm}[Covers and the fundamental group] \label{orbifold cover pi1}
There exists a natural one-to-one correspondence between subgroups $G \subset \piorb{X, \Delta}$ and covers $\pi \from Y \to X$ branched at most at $\Delta$.
Furthermore:
\begin{enumerate}
\item\label{ocp.1} $G$ is of finite index if and only if $\pi$ is finite.
\item\label{ocp.2} $G$ is a normal subgroup if and only if $\pi$ is Galois.
\item\label{ocp.3} Let $Y_{1, 2} \to X$ be two covers branched at most at $\Delta$, with corresponding subgroups $G_{1, 2} \subset \piorb{X, \Delta}$.
Then there is a factorization
\[ \begin{tikzcd}
 & & Y_2 \arrow[d] \\
Y_1 \arrow[rr] \arrow[urr, dashed, "\exists"] & & X
\end{tikzcd} \]
if and only if $G_1 \subset G_2$.
\end{enumerate}
\end{thm}

\begin{proof}
The proof is the same as in the snc case, cf.~\cite[Thm.~1.1]{Claudon08}, with one important difference:
in order to extend (possibly non-finite) \'etale covers of $X^*$ to branched covers of $X$, we would like to apply~\cite[Thm.~3.4]{DethloffGrauert94}.
In order to do this, we must invoke the finiteness of local orbifold fundamental groups of klt pairs, as proved in~\cite[Thm.~1]{Braun21}.
(Note that~\cite{Braun21} works in the algebraic category, but in view of~\cite[Thm.~1.7]{Fujino22} and~\cite[Rem.~6.10]{CGGN} his result extends to complex spaces as well.)
\end{proof}

\begin{defi}[Universal cover] \label{univ cov}
The (\emph{orbifold}) \emph{universal cover} of $(X, \Delta)$ is the cover $\pi \from \wt X_\Delta \to X$ corresponding to the trivial subgroup $\set1 \subset \piorb{X, \Delta}$ under the correspondence from \cref{orbifold cover pi1}.
\end{defi}

\noindent
Let $\wt \Delta$ be the divisor on $\wt X_\Delta$ which is supported on $\pi\inv(\supp \Delta)$ and satisfies
\[ K_{\wt X_\Delta} + \wt \Delta = \pi^*(K_X + \Delta). \]
It is easy to see that the pair $(\wt X_\Delta, \wt \Delta)$ is again klt with standard coefficients.
Also, $\wt \Delta = 0$ if and only if $\pi$ is branched exactly at $\Delta$.

\begin{defi}[Developable pairs]
We say that $(X, \Delta)$ is \emph{developable} if in the above notation, $\wt X_\Delta$ is smooth and $\wt \Delta = 0$.
\end{defi}

\noindent
Intuitively, being developable means that the universal cover is a manifold.

\begin{exa}
Consider the klt pair $(X, \Delta)$, where $X = \PP^1$ and
\[ \Delta = \left( 1 - \textstyle\frac1n \right) \cdot [0] + \left( 1 - \textstyle\frac1m \right) \cdot [\infty] \]
with $n, m \ge 2$.
Set $d = \operatorname{gcd}(n, m)$.
Then $\piorb{X, \Delta} = \factor \Z {d \Z}$, and the universal cover $\pi \from \wt X_\Delta = \PP^1 \to \PP^1$ is given by $[z_0 : z_1] \mapsto [z_0^d : z_1^d]$.
We have
\[ \wt \Delta = \left( 1 - \textstyle\frac1{n/d} \right) \cdot [0] + \left( 1 - \textstyle\frac1{m/d} \right) \cdot [\infty]. \]
In particular, $(X, \Delta)$ is developable if and only if $n = m$.
\end{exa}

\begin{cor}[Galois closure] \label{gal closure}
Let $Y \to X$ be a finite cover branched at most at $\Delta$.
Then there is a finite cover $Y' \to Y$ such that the composition $Y' \to X$ is finite, Galois, and branched at most at $\Delta$.
If additionally $Y \to X$ is branched exactly at $\Delta$, then the same is true of $Y' \to X$, and $Y' \to Y$ is \qe.
\end{cor}

\noindent
We call $Y' \to X$ the \emph{Galois closure} of $Y \to X$.

\begin{proof}
Using the correspondence from \cref{orbifold cover pi1}, the statement boils down to the following:
for a group $G$ and a subgroup $H \subset G$ of finite index, there is a normal subgroup $N \trianglelefteq G$ of finite index such that $N \subset H$.
But this is easy (and well-known): simply set
\[ N \defn \bigcap_{\mathclap{g \in G/H}} \, g H g\inv. \]
The last statement is easily seen to be true by comparing the ramification indices of $Y \to X$ and $Y' \to X$ over the components $\Delta_i$.
\end{proof}

\section{Orbifold Chern classes of klt pairs} \label{sec chern}

In this section, we recall the definition of the first and second orbifold Chern classes for klt pairs, in the spirit of~\cite{GK20}.
We then explain how to compute them concretely in two cases: in the projective setting by a cutting-down argument (\cref{subsec cutting down}), and when we have an ``orbi-resolution'' at our disposal (\cref{subsec orbi res}).

\subsection{The general \kahler case} \label{subsec general kahler c2}

Let us begin by recalling how to define Chern numbers associated with the first and second Chern classes.
This is nothing but a slight generalization of~\cite[Def.~5.2]{GK20} that takes into account the presence of a boundary.
The construction relies on the Chern--Weil formalism in the orbifold setting.
We will not recall the basic definitions and properties for the differential geometry of orbifolds (e.g.~Hermitian metrics on orbifold bundles, orbifold Chern classes, orbifold de Rham cohomology, and so on).
A good reference is~\cite[Sec.~2]{Blache96}.

\medskip
Let $(X,\Delta)$ as in \cref{setup:flat} and let $\Xc \subset X$ be the largest open subset of $X$ such that $(X, \Delta)$ admits a smooth orbi-\'etale orbi-structure $\cC^\circ$, and set $Z \defn X \setminus X^\circ$.
As proved in \cref{klt orb}, $\dim Z \le n - 3$.
Next, let $\alpha \in \HH{2n-4}.X.\R.$ where that cohomology space is understood as the cohomology of the locally constant sheaf $\underline{\R}_X$.
For dimensional reasons, we have an isomorphism $\HHc{2n-4}.\Xc.\R. \bij \HH{2n-4}.X.\R.$.
Next, the de Rham complex of orbifold differential forms on $\Xc$ yields a de Rham--Weil isomorphism $\HHdRc{\bullet}.\Xc, \R. \to \HHc{\bullet}.\Xc. \R.$, so that in the end we get a natural isomorphism
\begin{equation}\label{eq:cohomology 2n-4}
\psi: \HHdRc{2n-4}.\Xc, \R. \overset{\sim}{\longrightarrow} \HH{2n-4}.X.\R. .
\end{equation}

Now, let $E\to \Xc$ be an orbifold bundle for the pair $(X^\circ, \Delta^\circ)$. We can equip it with an orbifold Hermitian metric $h$ and form the Chern classes $\chorb{i}{E,h}$ which are orbifold differential forms of bidegree $(i,i)$.
We can use the isomorphism~\lref{eq:cohomology 2n-4} to define real numbers when $i=2$.
If $\alpha\in\HH{2n-4}.X.\R.$, the class $\psi^{-1}(\alpha)$ can be represented by a compactly supported orbifold $(2n-4)$-form $\Omega$ on $\Xc$, so that $\chorb2{E,h}\wedge \Omega$ is a compactly supported orbifold $(n,n)$-form on $\Xc$.

\begin{defi}\label{def:orbi Chern bundle}
The orbifold second Chern class $\ccorb2{E}$ is the unique element in the dual space $\HH{2n-4}.X.\R.\dual$ which under $\psi \dual$ corresponds to the Poincar\'e dual of the class $\chorb2{E}\in \HHdR{4}.\Xc, \R.$, where the latter is taken with respect to (but independent of) the orbi-structure $\cC^\circ$.
The quantity
\[\ccorb2{E} \cdot \alpha  \defn  \int_{X^\circ} \chorb2{E,h} \wedge \Omega\]
is thus a well defined real number for any class $\alpha \in \HH{2n-4}.X.\R.$.
\end{defi}

Let us apply the above construction to $\Omega^1_{(X^\circ,\Delta^\circ)}$ the orbifold bundle of differential forms. For the first Chern class, one can avoid the use of orbistructures and define it directly as a cohomology class as follows.

\begin{defi} \label{def:orbi c1}
For a klt pair $(X, \Delta)$, we set
\[ \ccorb1{X,\Delta} \defn \frac{1}{m}\cc1{\left({\can X}^{[m]}\otimes\O X(m\Delta)\right)\ddual}\in\HHs2.X,\R. \]
where $m \ge 1$ is an integer such that the reflexive rank~$1$ sheaf $\left({\can X}^{[m]}\otimes\O X(m\Delta)\right)\ddual$ is a line bundle.
\end{defi}

Now let us consider the case of the second Chern class.

\begin{defi}\label{def:orbi c2}
The orbifold second Chern class $\ccorb2{X, \Delta}\in \HH{2n-4}.X.\R.\dual$ of the pair $(X, \Delta)$ is the second Chern class of the orbi-bundle $\Omega^1_{(X^\circ,\Delta^\circ)}$ on $X^\circ$ defined in \cref{orbi cotang}. 
\end{defi}

\begin{rem} \label{c2 homology}
As already observed in \cite[p.~893]{GK20}, the object constructed in~\cref{def:orbi c2} is naturally a homology class:
\[ \ccorb2{X, \Delta} \in \Hh{2n-4}.X.\R. . \]
\end{rem}

\subsection{The projective case --- Mumford's construction} \label{subsec:Mumford}

Let $(X, \Delta)$ be a projective dlt pair with standard coefficients such that each component $\Delta_i$ of $\Delta$ is \Q-Cartier.
In~\cite[\S 3.1, p.~1458]{GT16}, the orbifold Chern classes $\ccorb2{X, \Delta}$ and $\cpcorb12{X, \Delta}$ were defined as multilinear forms on $\NS X$.
Here we would like to observe that this procedure can also be carried out without the assumption that the $\Delta_i$ be \Q-Cartier.
Our argument follows the proof of~\cite[Thm.~3.13]{GKPT15} closely. ---
We will restrict attention to the case of klt pairs, as we are only concerned with those in this paper.

So let $(X, \Delta)$ be an $n$-dimensional projective klt pair with standard coefficients.
Applying \cref{klt orb}, we obtain an open subset $X^\circ \subset X$ whose complement has codimension $\ge 3$ and such that $(X^\circ, \Delta|_{X^\circ})$ admits a smooth orbi-\'etale orbi-structure $\cC$.
Consider the ``big global cover'' $\gamma \from \wh{X^\circ} \to X^\circ$ associated to $\cC$, cf.~\cite[\S\S2--3]{Mumford83}, which up to shrinking $X^\circ$ may be assumed to be Cohen--Macaulay.
The locally free orbi-sheaf $\Omegar\cC1$ from \cref{orbi cotang} induces a genuine locally free sheaf $\sF$ on $\wh{X^\circ}$.
The Chern classes of~$\sF$ induce classes $\cc i{\Omegar\cC1} \in \Chow{X^\circ}{n - i}$.
Since $\Chow{X^\circ}*$ is equipped with a ring structure, we also have $\cpc1{2}{\Omegar\cC1} \in \Chow{X^\circ}{n - 2}$.
For dimensional reasons, $\Chow X{n - i} \bij \Chow{X^\circ}{n - i}$ is an isomorphism for $i \le 2$.
We obtain classes $\cc2{\Omegar\cC1}$ and $\cpc12{\Omegar\cC1} \in \Chow X{n - 2}$, which are independent of the choice of $\cC$ by~\cite[Prop.~3.5]{GT16}.
The orbifold Chern classes $\ccorb2{X, \Delta}$ and $\cpcorb12{X, \Delta}$ are then given by cap product with Chern classes of line bundles on $X$:
\begin{align*}
\ccorb2{X, \Delta} \cdot \sL_1 \cdots \sL_{n - 2} & \defn \deg \big( \cc2{\Omegar\cC1} \cap \cc1{\sL_1} \cap \cdots \cap \cc1{\sL_{n-2}} \big), \\
\cpcorb12{X, \Delta} \cdot \sL_1 \cdots \sL_{n - 2} & \defn \deg \big( \cpc12{\Omegar\cC1} \cap \cc1{\sL_1} \cap \cdots \cap \cc1{\sL_{n-2}} \big),
\end{align*}
and these maps factors via $\NS X$.

\subsection{The projective case --- cutting down} \label{subsec cutting down}

If $(X,\Delta)$ is a projective klt pair with standard coefficients, then \cref{klt orb} allows one to generalize Mumford's construction of $\Q$-Chern classes \cite{Mumford83} to this setting as explained above.
The fact that the Chern--Weil construction from~\cref{def:orbi c2} and Mumford's definition of $\Q$-Chern classes are equivalent is given in \cite[Claim~6.5]{GK20} in the case where $\Delta=0$.
It extends readily to the more general setting of klt pairs with standard coefficients.

Since $\psi$ is an abstract isomorphism, it is in practice difficult to actually compute these numbers.
There is, however, an important situation where things get much more explicit and that is when $\alpha =\cc1{L}^{n-2}$ where $L$ is an ample line bundle on $X$ (we could also have $(n-2)$ different ample line bundles, but let us stick to the former case for simplicity).
By homogeneity of the intersection product, we can assume that $L$ is very ample and induces an embedding $i:X\hookrightarrow \PP^N$ such that $L \isom i^* \O{\PP^N} (1)$.
We pick $(n-2)$ hyperplanes $H_1, \ldots, H_{n-2}$ in general position.
In particular, one has that $\sum H_i$ has simple normal crossings and $S \defn H_1\cap \cdots \cap H_{n-2}\cap X \subset X^\circ$. 

\begin{lem}\label{lem:c2 restriction surface}
With the notation as above, the Chern number from~\cref{def:orbi Chern bundle} can be computed with the following formula:
\begin{equation} \label{eq:restriction}
\ccorb2{E} \cdot \cc1{L}^{n-2} = \int_{S} \chorb2{E, h} \! \big|_S.
\end{equation}
\end{lem}

\begin{proof}
To begin with, let us choose sections $s_i\in \HH0.\PP^N.\cO_{\PP^N}(1).$ such that $H_i=\{s_i=0\}$, and we equip $\cO_{\PP^N}(1)$ with the Fubini--Study metric.
Next, we choose cut-off functions $\chi_i: \PP^N\to [0,1]$ such that 
\[\chi_i = 
\begin{cases}
0 & \mbox{on} \quad \set{|s_i|\le\delta} \\
1 & \mbox{on} \quad \set{|s_i|\ge2\delta}
\end{cases}
\]
for some $\delta>0$ small enough so that
\[ \bigcap_{i=1}^{n-2} \set{|s_i|\le 2\delta} \cap X \subset X^\circ. \]
For any $\ep\in (0,1]$, one defines $\vp_{i,\ep} \defn \chi_i \log |s_i|^2+(1-\chi_i)\log(|s_i|^2+\ep^2)$ and set $\omega_{i,\ep} \defn \om_{\rm FS}+dd^c \vp_{i,\ep}$.
Clearly, $\omega_{i,\ep}$ is supported on $\{|s_i| \le 2\delta\}$ and $\omega_{i,\ep} \to [H_i]$ as $\ep \to 0$, both weakly as currents on $\PP^N$ and locally smoothly away from $H_i$.
We set $\Omega_\ep  \defn  \bigwedge_{i=1}^{n-2}\omega_{i,\ep}$, which is supported on $\bigcap_{i=1}^{n-2} \{|s_i|\le 2\delta\}$.

The immersion $i: X^\circ \hookrightarrow \PP^N$ induces a commutative diagram
\[
 \begin{tikzcd}
 \HHdR{2n-4}.\PP^N, \R. \arrow[r, "\sim"] \arrow[d, "i^*"] &  \HH{2n-4}.\PP^N.\R.  \arrow[d, "i^*"] \\
 \HHdR{2n-4}.X^\circ, \R.   \arrow[r, "\sim"]&  \HH{2n-4}.\Xc.\R. .
 \end{tikzcd} 
\]
and by our choices the image $i_*[\Omega_\ep]$ lands in the image of the natural map
\[\HHdRc{2n-4}.X^\circ, \R. \to \HHdR{2n-4}.X^\circ, \R.\]
and satisfies $\psi(i_*[\Omega_\ep]) = \cc1{\cO_{\PP^N}(1)}^{n-2}|_X=\cc1{L}^{n-2}$. Therefore, we have for any $\ep>0$ the identity
\begin{equation}\label{eq:orbi chern class epsilon}
\ccorb2{E} \cdot \cc1{L}^{n-2} = \int_{X^\circ} \chorb2{E,h} \wedge \Omega_\ep.
\end{equation}
Now, since $\sum H_i$ has simple normal crossings, an easy local computation shows that $\Omega_\ep$ converges to the current of integration along the submanifold $W \defn \bigcap_{i=1}^{n-2} H_i$, both weakly on $\mathbb P^N$ and locally smoothly away from $W$.
Since the support of $\Omega_\ep|_X$ is contained in a fixed compact subset of $X^\circ$, ones sees that $\Omega_\ep|_{\Xc}$ converges weakly to $[S]=[ W \cap \Xc]$ in the sense of currents on the orbifold $X^\circ$.
Letting $\ep$ tend to~$0$ in~\lref{eq:orbi chern class epsilon}, we finally get the formula~\lref{eq:restriction}.
\end{proof}

\subsection{Orbi-resolutions and Chern numbers} \label{subsec orbi res}

When $X$ is smooth in codimension two, one can compute Chern numbers on a resolution of singularities, cf.~e.g.~\cite{CGG}.
In the presence of singularities in codimension two, it is explained in loc.~cit.~that a resolution does not compute Chern numbers anymore in general.
The substitute of a resolution in that setting is an \emph{orbi-resolution} as defined below.

\begin{defi}[Orbi-resolutions] \label{defi:orbi resolution}
Let $(X, \Delta)$ be a pair, where $X$ is a normal complex space, $\Delta$ has standard coefficients and let $\Xc \subset X$ be the orbifold locus of $(X, \Delta)$.
An orbi-resolution of $(X, \Delta)$ is a surjective, proper bimeromorphic map $\pi \from \wh X \to X$ from a normal complex space $\wh X$ such that:
\begin{enumerate}
\item $\big( \wh X, \wh \Delta \defn \pi_*^{-1}(\Delta) \big)$ has only quotient singularities, and
\item $\pi$ is isomorphic over $\Xc$.
\end{enumerate} 
\end{defi}

The existence of orbi-resolutions can be established\footnote{The proof of~\cite[Thm.~3]{LiTian19} applies verbatim when $\Delta \ne 0$, but we will only use the existence of orbi-resolutions when $\Delta = 0$.} for quasi-projective varieties (with $\Delta = 0$), using deep results about stacks as Chenyang Xu has showed in~\cite[\textsection 3]{LiTian19}.
However, the construction proposed there is highly non-canonical (or non-functorial) and this makes it difficult to generalize it to the complex analytic setting, even assuming algebraic singularities.

One important application of the existence of orbi-resolutions is highlighted by the following lemma, which shows that we can use such partial resolutions to compute the orbifold second Chern class of $(X, \Delta)$ against a class in $\HH2n-4.X.\R.$.

\begin{lem} \label{same c2}
Let $(X, \Delta)$ be a pair as in~\cref{setup:flat}.
Assume that $(X, \Delta)$ admits an orbi-resolution $\pi \from ( \wh X, \wh \Delta ) \to ( X, \Delta )$ as in~\cref{defi:orbi resolution}.
Given any $a \in \HH{2n-4}.X.\R.$, one has the formula
\begin{equation*}
\ccorb2{X,\Delta} \cdot a = \chorb2{\wh X, \wh \Delta} \cdot \psi(\pi^*a),
\end{equation*}
where on the right-hand side, $\chorb2{\wh X, \wh \Delta} \in \HHdR4.\wh X, \R.$ is the usual orbifold second Chern class of $(\wh X, \wh \Delta)$ and $\psi \from \HH\bullet.\wh X.\R. \to \HHdR{\bullet}.\wh X, \R.$ is the orbifold de Rham--Weil isomorphism.
\end{lem}

\begin{proof}
With the notation from~\cref{defi:orbi resolution}, let us denote $\wh X \setminus E \defn \pi^{-1}(\Xc)$ and $j \from \wh X \setminus E \to \wh X$ the natural inclusion; for simplicity we set $k \defn 2n-4$ and skip the reference to $\R$ in the cohomology spaces below.
Finally, we set  $\pi_0 \defn \pi|_{\wh X\setminus E} \from \wh X \setminus E \to \Xc$.

We then have the following diagram
\[
 \begin{tikzcd} 
 & & \HHdR k.\wh X. \\
\HHdRc{k}.{\wh X\setminus E}. \arrow[r,"\phi"] \arrow [urr, "j_*^\mathrm{dR}"] & \HHsc k.\wh X\setminus E.  \arrow[r,"j_*"] & \arrow[u, "\psi"] \HHs{k}.\wh X. \\
\HHdRc k.\Xc.  \arrow[r, "\phi"]  \arrow[u, "(\pi_0^\mathrm{dR})^*"]&  \HHsc k.\Xc. \arrow[r,"i_*"] \arrow[u, "\pi_0^*"] & \HHs k.X. \arrow[u,"\pi^*"]
 \end{tikzcd} 
\]
where all arrows except for $j_*, j_*^\mathrm{dR}$ and $\pi^*$ are isomorphisms. Now, one can pick an orbifold Hermitian metric $\wh h$ on $T_{\wh X,\wh \Delta}$ and descend it to an orbifold Hermitian metric $h$ on $T_{\Xc}$ since $\pi$ is an isomorphism $\wh X\setminus E \to \Xc$. Then, if as before $\alpha$ is an orbifold representative of $\phi^{-1}( i_*^{-1}(a))$ with compact support in $\Xc$, we have 
\begin{align*}
\ccorb2{X,\Delta}\cdot a &= \int_{\Xc} \chorb2{\Xc,h} \wedge \alpha\\
&=  \int_{\wh X\setminus E} \chorb2{\wh X,\wh h} \wedge \pi^*\alpha\\
&= \chorb2{\wh X,\wh \Delta} \cdot [\pi^*\alpha]_{\mathrm{dR}}\\
&=  \chorb2{\wh X,\wh \Delta} \cdot \psi(\pi^*a)
\end{align*}
since we have $\psi(\pi^*a) = (j_*)^\mathrm{dR}([\pi^*\alpha]_{\mathrm{dR}})$ from the commutativity of the diagram above. 
\end{proof}

We conclude this paragraph with a remark on the non-orbifold locus.
For the sake of clarity (and also since we will use only this case), we stick to the case $\Delta = 0$.

If $X$ is a normal complex space that admits an orbi-resolution $\pi \from \wh X \to X$ in the sense of~\cref{defi:orbi resolution}, it is immediate that its non-orbifold locus $X \setminus X^\orb$ coincides with $\pi(E)$, where $E \subset \wh X$ is the exceptional locus of $\pi$.
In particular, the non-orbifold locus is an analytic subset of $X$.
This latter statement is very natural and should be true regardless of the existence of orbi-resolutions.
Unfortunately, we are neither able to prove it in the general analytic setting nor able to locate a suitable reference.
We can, however, prove it under the additional assumption that the singularities of $X$ are algebraic.
This is sufficient for the application in \cref{sec flat case}.

\begin{lem}[Analyticity of the non-orbifold locus] \label{lem:orbifold locus analytic}
Let $X$ be a normal complex space having only algebraic singularities (in the sense of~\cite[Def.~2.4]{CGGN}).
Then its non-orbifold locus $Z \defn X \setminus X^\orb$ is a closed analytic subset.

In particular, this applies if $X$ is a compact klt \kahler space with $\cc1X = 0$.
\end{lem}

\begin{proof}
When $X$ is algebraic, this is a straightforward consequence of~\cite[Cor.~2.6]{Artin69}.
If $U \subset X$ is a euclidean open subset of $X$ being isomorphic through a map $\varphi \from U \bij V$ to an open subset $V \subset Y$ of an algebraic variety, then we have $\varphi(Z \cap U) = V \setminus V^\orb$, and this is an analytic subset of $V$ by the algebraic case.
The subset $Z \cap U$ is then given by the vanishing of a family of holomorphic functions, i.e.~it is analytic in $U$.

The last statement is a consequence of~\cite[Thm.~B]{BakkerGuenanciaLehn20}: $X$ can be realized as a member of a locally trivial family which also has projective fibers.
The family being locally trivial (over a smooth connected base), all the fibers are locally isomorphic and such an $X$ then has locally algebraic singularities (cf.~\cite[Ex.~2.5]{CGGN}).
\end{proof}

\section{Uniformization of canonical models} \label{sec MY eq}

In this section, we prove \cref{MY eq}.
Let us first introduce notation.
We set $A \defn K_X + \Delta$ and pick a complete intersection surface $S = D_1 \cap \cdots \cap D_{n - 2}$ of $n - 2$ general hypersurfaces $D_i \in |mA|$, where $m$ is sufficiently large and divisible. ---
The proof is divided into four steps.

\subsection*{Step~1: The orbi Higgs-sheaf $(\mathcal E_X,\theta_X)$}

Using the notation introduced in the proof of \cref{klt orb}, we can find a (a priori non-smooth) orbi-étale structure $\mathcal C=\{U_\alpha, g_\alpha, U_\alpha'\}$ with respect to $(X,\Delta)$ on the whole $X$. Then, one can define the reflexive orbi-Higgs sheaf $(\sE_{X}, \theta_{X})$ with respect to $\mathcal C$ as follows:
\begin{equation} \label{eq:orbi-higgs}
\theta_{X} \from \sE_{X} \defn \Omegar{(X, \Delta)}1 \oplus \O{(X, \Delta)} \lto  \sE_{X} \tensor \Omegar{(X, \Delta)}1,
\end{equation}
where on each chart $U'_\alpha$, we define $\theta_{U'_\alpha}(a, f) \defn (0, a)$ where $(a,f)$ is a section of $\sE_{U_\alpha'} \defn \Omega_{U_\a'}^{[1]}\oplus\O{U_\alpha'}$.
Cf.~also \cref{orbi cotang} and~\cite[\textsection 5.1, Step~2]{GT16}.

\medskip
In order to compute Chern numbers involving $\sE_X$, one needs to introduce a global cover $f \from Y \to X$ and an actual reflexive sheaf $\sE_Y$ on $Y$ as we now explain. Thanks to \cref{ex strictly adapted}, there exists a finite morphism $f \from Y \to X$ that is strictly adapted for $(X, \Delta)$ and whose extra ramification in codimension one (i.e.~away from $\supp(\Delta)$) is supported over a general element $H$ of a very ample linear system on $X$.
Let $N$ be the ramification order along $H$; we have
\begin{equation} \label{ram codim 1}
K_Y = f^* \! \left( K_X + \Delta + \left( 1 - \textstyle \frac1N \right) \! H \right).
\end{equation}

We set $D \defn \Delta + \big( 1 - \frac1N \big) H$ and define $(X, D)_\orb$ to be the largest open subset of $X$ where the pair $(X, D)$ admits a smooth orbi-\'etale orbi-structure $\cC^\circ$; we know that $\codim X { X \setminus (X, D)_\orb } \ge 3$ by \cref{klt orb}.
One can be a bit more precise about the shape of $\cC^\circ$, which will be useful later.
Recall from the proof of \cref{klt orb} that if we set $K \defn I\times J$ and $\alpha \defn (\beta,\gamma)\in K$, then we have a diagram 
\[
\begin{tikzcd}
X'_\a \arrow[r, "f_\a"] \arrow[rr, "h_\a", bend left]&   U_\a' \arrow[r,"g_\a"]  \arrow[d,hookrightarrow] &U_\a \arrow[d,hookrightarrow] \arrow[r,hookrightarrow] & X \arrow[d, "\mathrm{id}"] \\
& U_\beta' \arrow[r,"g_\beta"] & U_\beta \arrow[r,hookrightarrow] & X
\end{tikzcd}
\]
where $X'_{\a}$ is smooth and $f_\a$ is quasi-étale.
Note that one can ``restrict'' $\sE_X$ to the orbifold locus $\bigcup_\alpha U_\a\subset X$ of $(X,\Delta)$ to get a \emph{locally free} orbi-Higgs sheaf with respect to the smooth orbi-étale structure $\{U_\a, h_\a, X'_\a\}_{\alpha\in K}$ for the pair $(X,\Delta)$ in codimension two, given by $\sE_{X'_\a} \defn f_\a^{[*]}(\sE_{U'_\beta}|_{U'_\a}) \simeq \Omega_{X'_\a}^1\oplus \O{X'_\a}$.
In particular, one can define the Chern number $\ccorb2{\sE_X} \cdot A^{n-2}$ as explained in \cref{subsec general kahler c2}.

By choosing $H$ general, one can arrange that $h^*_\alpha H$ is smooth for all indices $\alpha \in K$ thanks to Bertini's theorem, so that a further Kawamata cover $\kappa_\a \from X_\a \to X'_\a$ orbi-étale with respect to $(X'_a, h_\a^*(1-\frac 1N)H)$ yields the expected smooth orbi-\'etale orbi-structure $\cC^\circ \defn \{U_\a, p_\a, X_\a\}_{\a\in K}$ for the pair $(X, D)$ in codimension two where $p_\a=h_a \circ \kappa_\a$. We end up with the following factorization:
\[ \begin{tikzcd}
    X_\a \arrow[rd, "\kappa_\a"'] \arrow[rr, "p_\alpha"] & & U_\alpha \arrow[rr, "\text{\'etale}"] & & X \\
  & X'_\a \arrow[ru, "h_\alpha"'] &
\end{tikzcd} \]

Next, set
\[ Y^\circ \defn f\inv \big( (X, D)_\orb \big) \cap (Y, \emptyset)_\orb \subset Y. \]
Since $f$ is finite, and by \cref{klt orb} applied to $(Y, \emptyset)$, we have $\codim Y {Y \setminus Y^\circ} \ge 3$.
The map $f$ restricts to $f^\circ \from Y^\circ \to \Xc \defn (X, D)_\orb$. 

Finally, we set $T \defn f\inv(S)$.
Since the linear system $|mA|$ (resp.~$f^*|mA|$) is basepoint-free and $S$ is general, we have $S \subset \Xc$ (resp.~$T \subset Y^\circ$).
Also, recall from \cref{surface} that $(S, D|_S)$ has quotient singularities.
The following diagram summarizes the situation:
\[ \begin{tikzcd}
  T \arrow[r, hookrightarrow] \arrow[d, "f|_T", swap] & Y^\circ \arrow[r, hookrightarrow] \arrow[d, "f^\circ"] & Y \arrow[d, "f"] \\
  S \arrow[r, hookrightarrow] & X^\circ \arrow[r, hookrightarrow] & X
\end{tikzcd} \]
Moreover, the ramification formula $K_T = f^*(K_S + D|_S)$ shows that $T$ is klt as well, i.e.~it is a surface with quotient singularities.

\subsection*{Step~2: Computing Chern numbers for $\sE_X$.}

Set $\Delta^\circ \defn \Delta|_{X^\circ}$ and $D^\circ \defn D|_{X^\circ}$.
Consider the locally free orbi-sheaf for the pair $(X^\circ, D^\circ)$ with respect to the orbi-structure~$\cC^\circ$ constructed in Step 1 above, defined by
\begin{equation} \label{1043}
\sE_{X_\alpha} = \Omegar{(X^\circ, \Delta^\circ, p_\alpha)}1 \oplus \O{X_\alpha}.
\end{equation}
Since $\big( X_\alpha, p_\alpha\inv (H) \big)$ is log smooth, the subsheaf $\Omegar{(X^\circ, \Delta^\circ, p_\alpha)}1 \subset \Omegap{X_\alpha}1$ has a very explicit expression in terms of local coordinates.
More precisely, if $(z_1, \ldots, z_n)$ is a local chart such that $p_\alpha\inv (H) = \set{z_1 = 0}$ on that chart, then the bundle at play is the subbundle of $\Omega^1_{X_\alpha}$ generated by $z_1^{N-1} \mathrm dz_1, \, \mathrm dz_2, \ldots, \mathrm dz_n$.
In particular, it agrees with $\Omega^1_{X_\alpha}$ outside of $p_\alpha\inv (H)$.

Now set $\sE_Y \defn \Omegar{(X, \Delta, f)}1 \oplus \O Y \subset \Omega^{[1]}_{Y}\oplus \O Y$, which we should think of as the reflexive pull back of $\sE_X$ by $f$. We equip this sheaf with the usual Higgs field~$\theta_Y$, and denote by $\sE_{Y^\circ}$ its restriction to $Y^\circ$.
Note that by~\lref{adl.2}, $\sE_Y = \Omegar Y1 \oplus \O Y$ holds on $Y \setminus f\inv(H)$.
Let $\big\{ ( V_\beta, q_\beta, Y_\beta ) \big\}_{\beta \in K}$ be a smooth orbi-\'etale (i.e.~quasi-\'etale, in this case) orbi-structure for $(Y^\circ, \emptyset)$, which exists by~\lref{esa.2} and \cref{klt orb} again, at least after shrinking $Y^\circ$.
Set $\sE_{Y_\beta} \defn q_\beta^{[*]} \sE_Y$ and consider the diagram
\begin{equation} \label{CD}
  \begin{tikzcd}
    W_{\alpha\beta} \arrow[rr, "r_{\alpha\beta}"] \arrow[dd, "g_\ab", swap] & & Y_\beta \arrow[d,"q_\beta"] \\
    & & Y^\circ \arrow[d,"f"] \\
    X_\alpha \arrow[rr, "p_\alpha"] & & X^\circ
  \end{tikzcd}
\end{equation}
where $W_{\alpha\beta}$ is the normalization of $X_\alpha\times_{X^\circ} Y_\beta$.
Since $p_\alpha$ is orbi-\'etale with respect to $D^\circ$, the map $r_{\alpha\beta}$ is \'etale over $\reg X^\circ \setminus \supp(D^\circ)$.
Moreover, since $q_\beta$ is quasi-étale, it follows that $f \circ q_\beta$ and $p_\alpha$ ramify to the same order along each component of $D$.
In other words, the smooth orbi-étale orbi-structures $\cC^\circ$ and $\big\{ \big( f(V_\beta), f \circ q_\beta, Y_\beta \big) \big\}$ are compatible.
In particular, $g_\ab$ and $r_\ab$ are étale so that $W_{\alpha\beta}$ is smooth, and we have additionally $g_\ab^* \sE_{X_\alpha} \isom r_\ab^* \sE_{Y_\beta}$ by~\lref{adl.3}.
Since $\sE_{X_\alpha}$ is locally free, so is $\sE_{Y_\beta}$, so that the reflexive sheaf $\sE_{Y^\circ}$ is a genuine orbifold bundle on the orbifold $Y^\circ$.

Let $\omega$ be an orbifold \kahler metric adapted to $(X^\circ, \Delta^\circ)$, as given by \cref{lem:kahler orbi}.
It is defined on an arbitrarily large relatively compact open subset of $X^\circ$.
In particular, it is defined in a neighborhood of $S$ and this will be enough for our purposes.
Set $S^* \defn \reg S \setminus \supp D$.
By definition, one has
\[ \ccorb2{\Omegar{(X, \Delta)}1 \big|_S} = \int_{\reg S \setminus \supp(\Delta)} \cc2{\Omegap{\reg X}1, \omega} = \int_{S^*} \cc2{\Omegap{\reg X}1, \omega} \]
and the last two integrals on the right are well-defined since $\omega$ pulls back to a smooth Kähler metric across points in $S_{\rm sing}\cup \supp(\Delta)$ via the finite maps $h_\alpha$.
The smooth form $p_\alpha^* \, \omega = f_\alpha^* \, h_\alpha^* \, \omega$ is semipositive, degenerate along $p_\alpha^{-1}(H)$.
More precisely, if $p_\alpha^{-1}(H) \cap U = \set{ z_1 = 0 }$ for some coordinate chart $U \subset X_\alpha$, then
\begin{equation*}
\begin{split}
p_\alpha^*\omega|_U = a_{1\bar 1} & |z_1|^{2(N-1)} idz_1\wedge d\bar z_1+\sum_{k=2}^n a_{1\bar k} z_1^{N-1} dz_1 \wedge i d\bar z_k +\\
& + \sum_{k=2}^n a_{k \bar 1} \bar z_1^{N-1} dz_k \wedge i d\bar z_1+ \sum_{j,k=2}^n a_{j\bar k} dz_j\wedge d\bar z_k
\end{split}
\end{equation*}
where $(a_{j\bar k})$ is smooth and definite positive. In particular, $p_\alpha^*\omega$ defines a smooth Hermitian metric on $\Omegar{(X^\circ, \Delta^\circ, p_\alpha)}1$.
Said otherwise, $g_{\alpha\beta}^* \, p_\alpha^* \, \omega$ induces a smooth Hermitian metric on $g_\ab^* \, \Omegar{(X^\circ, \Delta^\circ, p_\alpha)}1 \isom r_\ab^* \, \Omegar{(X^\circ, \Delta^\circ, f\circ q_\beta)}1$.
Hence, $q_\beta^* \, f^* \omega$ is a smooth Hermitian metric on the vector bundle $\Omegar{(X^\circ, \Delta^\circ, f \circ q_\beta)}1 = q_\beta^{[*]} \, \Omegar{(X^\circ, \Delta^\circ, f)}1$, so that $f^*\omega$ induces an orbifold metric on the orbi-bundle $\Omegar{(X^\circ, \Delta^\circ, f)}1$.
By the definition of the Chern classes of orbifold vector bundles, we have
\begin{align*}
\ccorb2{\Omegar{(X^\circ, \Delta^\circ, f)}1 \big|_T} & = \int_{f^{-1}(S^*)} \cc2{\Omega^1_{\reg Y}, f^*\omega} \\
& = \deg(f|_T) \, \cdot \, \int_{S^*} \cc2{\Omega^1_{\reg X}, \omega} \\
& = \deg(f) \, \cdot \,\ccorb2{\Omegar{(X, \Delta)}1 \big|_S}
\end{align*}
where the last identity follows from $\deg(f|_T) = \deg(f)$ since $S$ is general.
All in all, we find by~\cref{lem:c2 restriction surface}
\begin{equation}
\label{mult1}
\ccorb2{\sE_Y} \cdot (f^* A)^{n-2} = \deg(f) \, \ccorb2{\sE_X} \cdot A^{n-2}. 
\end{equation}
The same arguments show the similar identity 
\begin{equation}
\label{mult2}
\cpcorb12{\sE_Y} \cdot (f^* A)^{n-2} = \deg(f) \, \cpcorb12{\sE_X} \cdot A^{n-2}.
\end{equation}

\subsection*{Step~3: $(X, \Delta)$ has quotient singularities}

Consider on $X$ the orbi-Higgs sheaf $(\sF_X, \Theta_X) \defn \End(\sE_X, \theta_X)$.
It satisfies:
\begin{equation*} \label{MY equality}
\cpcorb12{\sF_X} \cdot A^{n-2} = \ccorb2{\sF_X} \cdot A^{n-2} = 0,
\end{equation*}
as follows from the assumption on the Chern classes of $(X, \Delta)$, i.e.~the assumption that equality holds in~\lref{eq:MY}.
Combined with \lref{mult1}--\lref{mult2}, the latter identity implies that the (genuine) Higgs sheaf $(\sF_Y, \Theta_Y) \defn \End(\sE_Y, \theta_Y)$ on $Y$ satisfies
\[ \cpcorb12{\sF_Y} \cdot (f^*A)^{n-2} = \ccorb2{\sF_Y} \cdot (f^*A)^{n-2} = 0. \]
Moreover, by~\cite[Sec.~4.4, proof of Thm.~C]{GT16}, the sheaf $\Omegar{(X, \Delta, f)}1$ is $(f^* A)$-semistable.
Recall that $\cc1{\Omegar{(X, \Delta, f)}1} = f^* A$ by~\cite[(3.11.5)]{ClaudonKebekusTaji21}.
It follows that $(\sE_Y, \theta_Y)$ is $(f^* A)$-Higgs-stable, cf.~the calculations in~\cite[proof of Cor.~7.2]{GKPT15}.
This in turn implies that the endomorphism sheaf $(\sF_Y, \Theta_Y)$ is $(f^* A)$-Higgs-polystable.
Indeed, the last assertion can be deduced from the usual smooth case by restricting to a general complete intersection curve and using the Mehta--Ramanathan theorem for Higgs sheaves~\cite[Thm.~5.22]{GKPT15}.
Cf.~also~\cite[Lemma~4.7]{GKPT20}.

By the Simpson correspondence for klt spaces~\cite[Thm.~5.1]{GKPT20}, the Higgs sheaf $(\sF_Y, \Theta_Y)\big|_{\reg Y}$ is locally free and is induced by a tame, purely imaginary harmonic bundle.
By~\cite[Prop.~3.17]{GKPT20}, the reflexive pull-back $g^{[*]} \sF_Y$ of $\sF_Y$ to a maximally quasi-\'etale cover $g \from Z \to Y$ (whose existence is guaranteed by~\cite[Thm.~1.5]{GKP16}) is locally free.

Now, set $W \defn X \setminus H \subset X$ and $h \defn f \circ g \from Z \to X$.
On $h\inv(W)$, we have that
\[ g^{[*]} \sE_Y \isom g^{[*]} \big( \Omegar Y1 \oplus \O Y \big) \isom \Omegar Z1 \oplus \O Z. \]
It follows that $g^{[*]} \sF_Y \isom \End \!\big( \Omegar Z1 \oplus \O Z \big)$, which contains the tangent sheaf $\T Z$ as a direct summand (again, only on $h\inv(W)$).
Since direct summands of locally free sheaves are locally free by Nakayama's lemma, the resolution of the Lipman--Zariski Conjecture for klt spaces~\cite{GKKP, GrafKovacs, Dru14} implies that $h\inv(W)$ is smooth.

By construction, the map $h\inv(W) \to W$ is branched exactly at $\Delta|_W$.
By \cref{gal closure}, its Galois closure $\wt W \to W$ also has this property, and $\wt W$ is smooth, being a \qe (hence \'etale) cover of the smooth space $h\inv(W)$.
This shows that $(W, \Delta|_W)$ has quotient singularities.
So far, we have only imposed that $H$ is general in its (basepoint-free) linear system.
We can therefore repeat the argument by choosing general elements $H_1, \ldots, H_{n+1} \in |H|$ and conclude that $(X, \Delta)$ has quotient singularities.
This means that $(X, \Delta)$ is a ``complex orbifold'' in the sense of~\cite[p.~109]{BG08}.

\subsection*{Step~4: $(X, \Delta)$ is a ball quotient}

Since $(X, \Delta)$ is a complex orbifold with $K_X + \Delta$ ample, there is an orbifold \kahler--Einstein metric $\omega$ such that $\Ric \om = -\om$, cf.~\cite[Thm.~5.2.2]{BG08}.
Set $X^* \defn \reg X \setminus \supp(\Delta)$, so that $\omega$ is a genuine \kahler metric on $X^*$.
One can compute the orbifold Chern classes using $\omega$, and, in particular, one has from the usual Chern form computations
\begin{align*}
0 &= \big( 2(n + 1) \, \ccorb2{X, \Delta} - n \, \cpcorb12{X, \Delta} \big) \cdot [K_X + \Delta]^{n - 2} \\
  &= \int_{X^*} \big( 2 (n + 1) \cc2{X, \omega} - n \cpc12{X, \omega} \big) \wedge \omega^{n-2} \\
  &=  C_n \int_{X^*} | \Theta^\circ(T_X, \omega) |_\omega^2 \, \omega^{n},
\end{align*}
where $C_n > 0$ is a dimensional constant, while
\[ \Theta^\circ(T_X, \omega) \defn \Theta(T_X, \omega) - \frac 1n \mathrm{tr}_{\mathrm{End}}(\Theta(T_X,\omega)) \cdot \id_{T_X} \]
is the trace-free Chern curvature tensor of $(T_X, \omega)$.

As a result, $\omega$ has constant negative bisectional curvature.
This implies that $\omega$ has negative Riemannian sectional curvature on $X^*$ by e.g.~\cite[\S 2.4.2]{Goldman99}.
(Note that one could also have said that $(X^*, \omega)$ is locally isometric to the complex hyperbolic space $(\B^n, \omega_{\mathrm{hyp}})$ by~\cite[Thm.~6]{Bochner47} and conclude by the usual curvature properties of the complex hyperbolic metric.)

Let $\pi \from \wt X_\Delta \to X$ be the orbifold universal cover of $(X, \Delta)$, cf.~\cref{univ cov}.
By the previous paragraph, $(X, \Delta, \omega)$ is an orbifold of nonpositive Riemannian sectional curvature.
It then follows from~\cite[Cor.~2.16 on p.~603]{BH99} that $(X, \Delta)$ is developable.
Now, $(\wt X_\Delta, \pi^* \omega)$ is a simply connected \kahler manifold with constant negative bisectional curvature, so it is holomorphically isometric to $(\B^n, \omega_{\mathrm{hyp}})$ by~\cite[Thm.~7.9]{KN69II}.
In particular, $\wt X_\Delta \isom \B^n$, proving \cref{MY eq}. \qed

\section{Characterization of ball quotients} \label{sec proof corollary}

In this section, we prove \cref{main3}.
We prove the implications~\lref{main3.1} $\imp$~\lref{main3.2} $\imp$~\lref{main3.3} $\imp$~\lref{main3.1} separately.

\subsection*{\lref{main3.1} $\imp$ \lref{main3.2}}

This is \cref{MY eq}.

\subsection*{\lref{main3.2} $\imp$ \lref{main3.3}}

Let $\pi \from \B^n \to X$ be the orbifold universal cover of $(X, \Delta)$.
(In particular, $(X, \Delta)$ is developable.)
By~\lref{ocp.2}, the map $\pi$ is Galois, with Galois group $\Gamma \isom \piorb{X, \Delta}$.
Note that $\Gamma \subset \Aut(\B^n) = \mathrm{PU}(1, n)$ is a finitely generated linear group.
Furthermore, the stabilizers of the action $\Gamma \acts \B^n$ are finite by~\lref{743}.
By Selberg's lemma~\cite{Alperin87}, there is a finite index normal subgroup $\Gamma' \subset \Gamma$ which is torsion-free.
This implies that $\Gamma'$ acts freely on $\B^n$.
We obtain the following factorization of~$\pi$:
\[ \B^n \xrightarrow{\quad\;\;\quad} \factor{\B^n}{\Gamma'} \xrightarrow{\quad f \quad} \factor{\B^n}{\Gamma} = X, \]
where $f$ is the quotient by the action of the finite group $G \defn \factor{\Gamma}{\Gamma'}$ on the projective manifold $Y \defn \factor{\B^n}{\Gamma'}$.
Since the first map is \'etale, it exhibits $\B^n$ as the universal cover of $Y$.
Combining this with the fact that $\pi$ is branched exactly at $\Delta$, we infer that $f$ is orbi-\'etale.

\subsection*{\lref{main3.3} $\imp$ \lref{main3.1}}

Recall that $K_Y$ is ample and that $Y$ satisfies equality in the Miyaoka--Yau inequality, cf.~e.g.~\cite[(8.8.3)]{Kollar95}.
As $f \from Y \to X$ is orbi-\'etale, it follows that also $K_X + \Delta$ is ample and equality likewise holds in the Miyaoka--Yau inequality for $(X, \Delta)$. \qed

\section{Uniformization of minimal models} \label{sec Q-Cartier}

This section has two (related) purposes:
first, to remove the assumption about the irreducible components of $\Delta$ being \Q-Cartier from \cref{th:MY ineq}.
And second, to prove \cref{T:general type}.

\subsection{Orbifold Miyaoka--Yau inequality}

In \cref{th:MY ineq}, or more generally in~\cite[Thm.~B]{GT16}, the assumption that the $\Delta_i$ be \Q-Cartier can be dropped without replacement.
We give two proofs of this result, the first one relying on~\cite{BCHM} and the second one on~\cref{ex strictly adapted}.

\begin{thm}[Miyaoka--Yau inequality] \label{GT16 B}
Let $(X, \Delta)$ be an $n$-dimensional projective klt pair with standard coefficients, and assume that $K_X + \Delta$ is big and nef.
Then the following inequality holds:
\begin{equation}
\big( 2(n + 1) \, \ccorb2{X, \Delta} - n \, \cpcorb12{X, \Delta} \big) \cdot [K_X + \Delta]^{n - 2} \ge 0.
\end{equation}
\end{thm}

\begin{proof}[First proof]
Consider a \Q-factorialization $f \from X' \to X$, cf.~\cite[Cor.~1.4.3]{BCHM} applied with $\fE = \emptyset$.
Set $\Delta' \defn f\inv_* \Delta$.
The map $f$ is small, meaning that $\Exc(f) \subset X'$ has codimension at least two.
Therefore $(X', \Delta')$ reproduces all the assumptions made on $(X, \Delta)$, and in addition $X'$ is \Q-factorial.
In particular, $K_{X'} + \Delta' = f^* ( K_X + \Delta )$ is big and nef.
Furthermore, $f(\Exc(f)) \subset X$ has codimension $\ge 3$, therefore $f_* \big( \ccorb2{X', \Delta'} \! \big) = \ccorb2{X, \Delta}$ as homology classes, and likewise for $\cpcorb12{X', \Delta'}$ (cf.~\cref{c2 homology}).
By the projection formula, we obtain
\begin{equation*}
\begin{split}
\big( 2(n + 1) \, & \ccorb2{X, \Delta} - n \, \cpcorb12{X, \Delta} \big) \cdot [K_X + \Delta]^{n - 2} = \\
& \big( 2(n + 1) \, \ccorb2{X', \Delta'} - n \, \cpcorb12{X', \Delta'} \big) \cdot [K_{X'} + \Delta']^{n - 2}.
\end{split}
\end{equation*}
The right-hand side is non-negative by~\cite[Thm.~B]{GT16}.
\end{proof}

\begin{proof}[Second proof]
Observe that in~\cite{GT16}, the assumption that the $\Delta_i$ be \Q-Cartier is only used in order to construct a strictly adapted morphism whose extra ramification is supported on a general very ample divisor (cf.~Ex.~2.11 of that paper).
However, using \cref{ex strictly adapted} we can construct such a cover even without that assumption.
After that, the proof of~\cite[Thm.~B]{GT16} applies verbatim.
\end{proof}

\subsection{Uniformization of minimal models}

In order to prove \cref{T:general type}, we use the strategy explained in~\cite[Step~1, p.~1086]{GKPT20}.
This means we first have to prove the following lemma.

\begin{lem} \label{can mod eq}
In the setting of \cref{T:general type}, the canonical model $(X_\canmod, \Delta_\canmod)$ also satisfies equality in~\lref{eq:MY}.
\end{lem}

Assuming \cref{can mod eq} for the moment, we then apply \cref{MY eq} on $(X_\canmod, \Delta_\canmod)$ to conclude.
This finishes the proof of \cref{T:general type}.

\begin{rem} \label{qcartier}
If we had proved \cref{MY eq} only in the setting of~\cite{GT16} (that is, assuming that the $\Delta_i$ are \Q-Cartier), then the above argument would break down.
This is because the irreducible components of $\Delta_\canmod$ may not be \Q-Cartier (even if the same is true of $\Delta$).
\end{rem}

\begin{proof}[Proof of \cref{can mod eq}]
As in the statement of \cref{T:general type}, let $(X_\canmod, \Delta_\canmod)$ denote the canonical model of the pair $(X, \Delta)$ and $\pi \from (X, \Delta) \to (X_\canmod, \Delta_\canmod)$ the canonical morphism ($K_X + \Delta$ being big and nef, some multiple is basepoint-free and so $\pi$ is a morphism).
By construction, $K_{X_\canmod} + \Delta_\canmod$ is ample and $\pi$ is crepant:
\begin{equation} \label{eq:crepant}
K_X + \Delta = \pi^* \big( K_{X_\canmod} + \Delta_\canmod \big).
\end{equation}
The pair $(X_\canmod, \Delta_\canmod)$ still has klt singularities.
From \cref{th:MY ineq}, we know that the inequality~\lref{eq:MY} holds for $(X_\canmod, \Delta_\canmod)$ and we are led to checking that:
\begin{equation} \label{eq:MY can model}
\begin{split}
\big( 2(n + 1) \, &\ccorb2{X, \Delta} - n \, \cpcorb12{X, \Delta} \big) \cdot [K_X + \Delta]^{n - 2} \ge \\
& \big( 2(n + 1) \, \ccorb2{X_\canmod, \Delta_\canmod} - n \, \cpcorb12{X_\canmod, \Delta_\canmod} \big) \cdot [K_{X_\canmod} + \Delta_\canmod]^{n - 2}.
\end{split}
\end{equation}
In view of~\lref{eq:crepant}, this amounts to showing
\begin{equation} \label{eq:c_2 can model}
\ccorb2{X, \Delta} \cdot [K_X + \Delta]^{n - 2} \ge
\ccorb2{X_\canmod, \Delta_\canmod} \cdot [K_{X_\canmod} + \Delta_\canmod]^{n - 2}.
\end{equation}
At this point, let us consider a general surface $\Sigma \subset X_\canmod$ cut out by the linear system $\vert m ( K_{X_\canmod} + \Delta_\canmod ) \vert$ (for $m \gg 1$ sufficiently divisible) and let us look at its preimage $S \defn \pi\inv(\Sigma) \subset X$ in~$X$.
The pairs\footnote{To avoid cumbersome notation, the restriction of the divisors $\Delta$ and $\Delta_\canmod$ to $S$ and $\Sigma$ is not written out.} $(S, \Delta)$ and $(\Sigma, \Delta_\canmod)$ are orbifold surfaces and contained in the orbifold loci of $(X, \Delta)$ and $(X_\canmod, \Delta_\canmod)$ respectively.
Obviously, $(\Sigma, \Delta_\canmod)$ is nothing but $(S, \Delta)_\canmod$ and we can apply~\cite[Thm.~4.2]{Megyesi}.
This yields
\[ 4 \, \ccorb2{\Sigma, \Delta_\canmod} - \cpcorb12{\Sigma, \Delta_\canmod} \le 4 \, \ccorb2{S, \Delta} - \cpcorb12{S, \Delta}. \]
The morphism $\pi|_S \from (S, \Delta) \to (\Sigma, \Delta_\canmod)$ being crepant, the above inequality reads as
\begin{equation} \label{eq:megyesi}
\ccorb2{\Sigma, \Delta_\canmod} \le \ccorb2{S, \Delta}.
\end{equation}
With the notation introduced, the inequality~\lref{eq:c_2 can model} boils down to the following:
\[ \ccorb2{ \cT_{(X, \Delta)} \big|_S } \ge \ccorb2{ \cT_{(X_\canmod, \Delta_\canmod)} \big|_\Sigma }. \]
This last inequality can be checked as in~\cite[pp.~1086--1087]{GKPT20} by considering the (orbifold) normal sequences
\begin{align}
0 \lto \cT_{(S, \Delta)} \lto & \cT_{(X, \Delta)} \big|_S \lto \cN_{(S, \Delta) \mid (X, \Delta)} \lto 0, \label{eq:normal seq1} \\
0 \lto \cT_{(\Sigma, \Delta_\canmod)} \lto & \cT_{(X_\canmod, \Delta_\canmod)} \big|_\Sigma \lto \cN_{(\Sigma, \Delta_\canmod) \mid (X_\canmod, \Delta_\canmod)} \lto 0. \label{eq:normal seq2}
\end{align}
It is worth noting that both sequences~\lref{eq:normal seq1} and~\lref{eq:normal seq2} are exact sequences of orbifold vector bundles, since the surface $S$ (resp.~$\Sigma$) is contained in the orbifold locus of $(X, \Delta)$ (resp.~$(X_\canmod, \Delta_\canmod)$) and the terms in the middle are thus genuine orbifold bundles.
Now it is enough to remark that the normal bundles $\cN_{(S, \Delta) \mid (X, \Delta)}$ and $\cN_{(\Sigma, \Delta_\canmod) \mid (X_\canmod, \Delta_\canmod)}$ satisfy
\begin{equation} \label{eq:normal pull back}
\cN_{(S, \Delta) \mid (X, \Delta)} \isom \pi^* \! \left( \cN_{(\Sigma, \Delta_\canmod) \mid (X_\canmod, \Delta_\canmod)} \right).
\end{equation}
Together with~\lref{eq:crepant} and~\lref{eq:megyesi}, this finally proves that the inequality~\lref{eq:c_2 can model} holds true.
This concludes the proof of \cref{can mod eq}.
\end{proof}

\begin{rem-plain}
In general, the canonical morphism $\pi|_S \from (S, \Delta) \to (\Sigma, \Delta_\canmod)$ is \emph{not} an orbifold morphism, but the normal bundles are actual locally free sheaves defined on $S$ (resp.~on $\Sigma$) and not only on the orbifold $(S, \Delta)$ (resp.~$(\Sigma, \Delta_\canmod)$).
The Chern classes of $\cN_{(\Sigma, \Delta_\canmod) \mid (X_\canmod, \Delta_\canmod)}$ thus come from $\Sigma$ and can be pulled back to $S$ in the usual way.
\end{rem-plain}

\section{Characterization of torus quotients} \label{sec flat case}

In this final section, we first establish the positivity of the orbifold second Chern class for Calabi--Yau and for irreducible holomorphic symplectic varieties.
Using the Decomposition Theorem~\cite{BakkerGuenanciaLehn20}, we can then easily deduce \cref{th:MY singular Kahler} and \cref{th:torus quotients}.
Finally, we prove \cref{th:characterization torus quotient}.

\subsection{Positivity of the second Chern class --- the projective case}

If $X$ is projective, then we know that it has an orbi-resolution in the sense of \cref{defi:orbi resolution}, and we can use this to understand the orbifold second Chern class of $X$.

\begin{prop} \label{cy}
Let $X$ be a projective irreducible Calabi--Yau (resp. irreducible holomorphic symplectic) variety of dimension $n$ with klt singularities and let $\beta \in \HH2.X.\R.$ be a \kahler class.
Then we have
\[ \ccorb2X \cdot \beta^{n-2} > 0. \]
\end{prop}

\begin{proof}
Let $\pi \from \wh X \to X$ be an orbi-resolution, whose existence is garanteed by \cite{LiTian19} since $X$ is projective. Let $\wh \beta$ be a \kahler class on $\wh X$ and let $\omega \in \beta$ (resp. $ \wh \omega \in \wh \beta$) be a \kahler form. 
 Recall that it follows easily from the Bochner principle \cite[Thm.~A]{CGGN}  that $T_X$ is stable with respect to $\beta$. This implies that $T_{\wh X}$ is stable with respect to $\pi^*\beta$, hence $T_{\wh X}$ is stable with respect to $\pi^*\beta+\ep\wh \beta$ for $\ep>0$ small enough, cf e.g. \cite[Prop.~3.4]{CGG}. In particular, as explained in \cite[Thm.~4.2]{EyssidieuxSala}, there exists an orbifold Hermite--Einstein metrics $h_\ep$ on $T_{\wh X}$ with respect to $\omega_\ep  \defn  \pi^*\omega + \ep \wh \omega$. From \cref{same c2}, we have
\[\ccorb2{X}\cdot \beta^{n-2}= \lim_{\ep \to 0} \int_{\wh X}\cpc2{\mathrm{orb}}{T_{\wh X}, h_\ep} \wedge \omega_\ep^{n-2}. \]
The exact same arguments as in \cite[Prop.~3.11]{CGG} using orbifold forms instead of usual forms shows that the latter quantity is non-negative, and if it is zero, then we have $\ccorb2{X}\cdot \gamma^{n-2}=0$ for \emph{any} \kahler class $\gamma$ on $X$. We claim that this cannot happen. Indeed, since $X$ is projective, this applies to classes of the form $\cc1{H}$ for an ample divisor $H$ on $X$. Then \cite{LuTaji18} would imply that $X$ is the quotient of an Abelian variety, clearly a contradiction.
\end{proof}

\subsection{Positivity of the second Chern class --- the IHS case}

We will derive the general \kahler case from the projective one using a deformation argument, as in~\cite[Prop.~4.4]{CGG}.

\begin{prop} \label{ihs}
Let $X$ be an irreducible holomorphic symplectic variety of dimension $n$ with klt singularities and let $\beta \in \HH2.X.\R.$ be a \kahler class.
Then we have 
\[ \ccorb2X \cdot \beta^{n-2} > 0. \]
\end{prop}

\begin{proof}
We will first prove that there exists a constant $C_X\in \mathbb R$ such that 
\begin{equation} \label{formula c2}
\ccorb2{X}\cdot a = C_X q_X(a)^{\frac n2-1}
\end{equation}
for any $a \in \HH2.X.\R.$, where $q_X \from \HH2.X.\R.\to \C$ is the Beauville--Bogomolov--Fujiki quadratic form.
Moreover, we will see that $C_X$ is constant when $X$ moves in a locally trivial family. 

The result follows from standard arguments (see e.g. \cite[Prop.~4.4]{CGG} and references therein) once one has proved that the formation of $\ccorb2{X} \cdot a$ is invariant under parallel transport along a locally trivial deformation, which we now prove. 

Let $\pi\from \mathfrak X \to \mathbb D$ be a proper surjective map which is a locally trivial deformation of $X=\pi^{-1}(0)$.
We denote by $\mathfrak X^{\rm orb}$ (resp. $X_t^{\rm orb}$) the orbifold locus of $\mathfrak X$ (resp. $X_t$), which is a Zariski open subset of $\mathfrak X$ (resp. $X_t$) according to \cref{lem:orbifold locus analytic}.
Next, we set $Z \defn \mathfrak X \setminus \mathfrak X^{\rm orb}$ and $Z_t=Z\cap X_t$.
The family being locally trivial, we infer that $\mathfrak X^{\rm orb}\cap X_t=X_t^\mathrm{orb}$ and thus that $Z_t=X_t\setminus X_t^\mathrm{orb}$.

\begin{claim}
\label{claim}
Up to shrinking $\mathbb D$, there exists a $\mathcal C^\infty$ diffeomorphism $F: \mathfrak X \to X_0\times \mathbb D$ commuting with the projection to $\mathbb D$ such that
\begin{enumerate}[label=(\roman*)]
\item\label{flow.1}  $F$ preserves the orbifold locus, i.e. $F(X_t^{\rm orb})= X_0^{\rm orb}\times \{t\}$.
\item\label{flow.2} $F|_{X_t^{\rm orb}}\from X_t^{\rm orb} \to X_0^{\rm orb}$ is smooth in the orbifold sense.
\end{enumerate}
\end{claim}
In this singular context, we mean that $F$ is the restriction of a smooth map under local embeddings in $\mathbb C^N$ which induces an homeomorphism between $\mathfrak X$ and  $ X_0\times \mathbb D$.

\begin{proof}[Proof of \cref{claim}]
Let us start with the existence of the diffeomorphism $F$. To do so, one can find a proper $\mathcal C^\infty$ embedding $\iota \from \mathfrak X \hookrightarrow \mathbb C^N$ thanks to \cite{ABT79}. Next,  extend $\pi$ smoothly to a smooth map $f$ with support in a neighborhood of $\iota(X)$. Since $\pi:\mathfrak X\to \mathbb D$ is locally trivial, one can stratify $\mathfrak X$ such that the restriction of $\pi$ to each stratum is proper and smooth (in the analytic sense, i.e. it is a submersion). The existence of $F$ then follows from Thom's first isotopy lemma, cf \cite[Prop.~11.1]{Mather70}.

In order to prove the two items in the claim, let us briefly recall the construction of $F$ in \emph{loc.~cit.} while emphasizing on the important points for our purposes. Start with local holomorphic trivializations $g_\alpha: U_\alpha  \to (U_\alpha \cap X_0)\times \mathbb D$ for a covering of analytic open sets $(U_\alpha)_{\alpha\in A}$ of $\mathfrak X$, and let $Z=\sqcup Z^{(k)}$ be the standard stratification of the analytic set $Z\subset \mathfrak X$.
The maps $g_\alpha$ induces a local biholomorphism between $Z^{(k)}$ and $Z_0^{(k)}\times \mathbb D$ for all $k$; in particular the holomorphic vector fields $v_\alpha \defn g_\alpha^* \frac{\partial}{\partial t}$ satisfy
\[v_\alpha\big|_{Z^{(k)}} \in \HH0.Z^{(k)}.\T {Z^{(k)}}.\]

Next, let $(\chi_\alpha)$ be a partition of unity subordinate to the open cover $(U_\alpha)_{\alpha\in A}$.
The $\mathcal C^\infty$ vector field $v \defn \sum \chi_\alpha v_\alpha$ still satisfies 
\[v|_{Z^{(k)}} \in \mathcal C^\infty(Z^{(k)}, T_{Z^{(k)}}).\]
As showed in \cite{Mather70}, its flow $(F_t)$ is well-defined over $\pi^{-1}(\mathbb D_{1/2})$ for $|t|<1/2$, and it preserves $Z^{(k)}$ for all $k$, hence it preserves $Z$ as well. Equivalently, the flow of $v$ preserves $\mathfrak X^{\rm orb}$, which proves~\lref{flow.1}.

Moreover, $v|_{\mathfrak{X^{\rm orb}}}$ is smooth in the orbifold sense (i.e.~when pulled back to the local smooth covers), a property which need not be true for arbitrary vector fields.
This is straightforward since the $v_\alpha$ satisfy this property (they lift to holomorphic vector fields on the \qe local covers), and multiplying by smooth functions is harmless. In order to prove~\lref{flow.2}, let $x_0\in X_0^{\rm orb}$ be an arbitrary point and let  $U\subset \mathfrak X^{\rm orb}$ be a small connected open neighborhood of  $x_0$ admitting a smooth quasi-étale cover $p:\wh U\to U$. We can find $U'\Subset U$ such that for $|t|\le s$ (with $s>0$ small enough) the flow $F_t$ is defined on $U$ and satisfies $F_t(U')\subset U$. Remember that $\wh v \defn p^*v|_{\reg U}$ extends to a smooth vector field on $\wh U$ which we still denote by $\wh v$, and whose flow we denote by $\wh F_t$. Since $p$ is étale over $\reg U$, uniqueness of flow ensures that we have a commutative diagram
\[
 \begin{tikzcd} 
p^{-1}(U')  \arrow[d, "p"] \arrow[r,"\wh F_t"] & p^{-1}(F_t(U')) \arrow[d,"p"] \\
U'   \arrow[r, "F_t"]  & F_t(U').
 \end{tikzcd} 
\]
Indeed, since $p$ is a local diffeomorphism over $\reg U$, we get
\[F_t \circ p = p \circ \wh F_t\ \text{on}\ p^{-1}(\reg U),\]
hence everywhere by continuity of the above maps. In summary, $F_t:U'\to F_t(U')$ is an homeomorphism which therefore lifts to the diffeomorphism $\wh F_t$ between the manifolds  $p^{-1}(U')$ and its image $p^{-1}(F_t(U'))$.  That is, $F_t$ induces an orbifold diffeomorphism between $U'$ and $F_t(U')$. Item~\lref{flow.2} is now proved. 
\end{proof}

Let us now consider the orbifold diffeomorphisms $F_t^{\rm orb}: X_t^{\rm orb} \to X_0^{\rm orb}$, and let $h_0$ be an orbifold Hermitian metric on $T_{X_0^{\rm orb}}$. Finally, let $\alpha_0$ be a closed orbifold form with compact support on $X_0^\orb$ representing a class $a_0\in \HH{2n-4}.X_0.\R.$. We have  
\begin{align*}
\ccorb2{X_0} \cdot a_0 & = \int_{X_0^{\rm orb}} \cpc2{\mathrm{orb}}{X_0^{\rm orb},h_0} \wedge \alpha_0 \\
&= \int_{X_t^{\rm orb}} (F_t^{\rm orb})^*\left(\cpc2{\mathrm{orb}}{X_0^{\rm orb},h_0} \wedge \alpha_0\right) \\
&= \int_{X_t^{\rm orb}}\cpc2{\mathrm{orb}}{X_t^{\rm orb}, (F_t^{\rm orb})^*h_t} \wedge (F_t^{\rm orb})^*\alpha_0\\
&= \ccorb2{X_t} \cdot F_t^*a_0
\end{align*}
where the last line comes from the fact that we have a commutative diagram
\[ \begin{tikzcd}
\HHdRc{2n-4}.X_t^{\rm orb}, \C.  \arrow[r,"\sim"] & \HH{2n-4}.X_t.\C. \\
\HHdRc{2n-4}.X_0^{\rm orb}, \C.  \arrow[u, "(F_t^{\rm orb})^*"] \arrow[r, "\sim"]  & \HH{2n-4}.X_0.\C. \arrow[u,"F_t^*"]
\end{tikzcd} \]
so that~\lref{formula c2} is proved.

Finally, we must show that $C_X > 0$.
Since $C_X$ is invariant under locally trivial deformation, one can use~\cite[Cor.~1.3]{BL18} and~\cite[Cor.~3.10]{BakkerGuenanciaLehn20} to deform $X$ locally trivially to a projective IHS variety $Y$.
\cref{cy} shows that $C_Y > 0$, which concludes the proof of the proposition.
\end{proof}

\subsection{Simultaneous proof of~\cref{th:MY singular Kahler} and~\cref{th:torus quotients}}

Here we closely follow the arguments from~\cite[proof of Thm.~5.2]{CGG}.

Let $(X, \Delta)$ be as in \cref{setup:flat} and such that $\ccorb1{X,\Delta}=0$. We denote by $X^\circ \defn (X,\Delta)_{\rm orb}$ the open locus where the pair has quotient singularities, and set $\Delta^\circ \defn \Delta|_{X^\circ}$.  It has been proved in \cite[Cor.~1.18]{JHM2} that abundance holds for such a pair and in particular $K_X+\Delta$ is torsion. We can then apply \cref{shokurov} and infer the existence of an orbi-\'etale map $f\from Y\to X$ such that
\[\O Y\isom K_Y\isom f^*(K_X+\Delta).\]
Arguing as in the proof of formula~\lref{mult1}, one has:

\begin{lem}
We have the identity
\begin{equation} \label{eq:multiplicativity c2}
\ccorb2Y \cdot f^*(\alpha)^{n-2} = \deg(f) \, \ccorb2{X,\Delta} \cdot \alpha^{n-2}.
\end{equation}
\end{lem}

\begin{proof}
Let $a$ be an orbifold differential form of degree $2n-4$ with compact support in $X^\circ$ representing $\alpha^{n-2}$ and let $h$ be an orbifold Hermitian metric on $\Omega^1_{(X^\circ,\Delta^\circ)}$. Consider the space $Y^\circ=f^{-1}(X^\circ)$; by taking a fiber product with local smooth charts of $X^\circ$, it follows easily from purity of branch locus that $Y^\circ$ admits a smooth orbistructure and that $f^*h$ induces an smooth Hermitian metric on $\Omega_{Y^\circ}$. In particular, we have 
\begin{align*}
\ccorb2{Y}\cdot f^*(\alpha)^{n-2}&=\int_{Y^\circ}\cc2{\Omega_{Y^\circ},f^*h} \wedge f^*a\\
&=\int_{Y^\circ\setminus f^{-1}(\supp \Delta)}\cc2{\Omega_{Y^\circ},f^*h} \wedge f^*a\\
&=\deg(f)\, \int_{X^\circ \setminus \supp \Delta}\cc2{\Omega_{(X^\circ,\Delta^\circ)},h} \wedge a\\
&=\deg(f)\, \int_{X^\circ }\cc2{\Omega_{(X^\circ,\Delta^\circ)},h} \wedge a\\
&=\deg(f)\, \ccorb2{X,\Delta}\cdot \alpha^{n-2},
\end{align*}
which proves the lemma.
\end{proof}

Both members of the equation~\lref{eq:multiplicativity c2} being simultaneously non-negative or zero (and $f^*(\alpha)$ still being a \kahler class on $Y$), we shall replace $X$ with $Y$ and assume from now on that there is no orbifold structure in codimension one, i.e.~that $\Delta = 0$.

By \cite[Thm.~A]{BakkerGuenanciaLehn20}, there exists a finite, Galois quasi-étale cover $f\from X'\to X$ such that $X' \isom T\times \prod_{i\in I} Y_i \times \prod_{j\in J} Z_j$ where $T$ is a torus, $Y_i$ are CY varieties and $Z_j$ are IHS varieties.
By \cite[Prop.~5.6]{GK20}, we have 
\[\ccorb2{X'} \cdot f^*\beta^{n-2} = \deg(f) \,\ccorb2{X} \cdot \beta^{n-2},\] 
while $f^*\beta$ is still a \kahler class by \cite[Prop.~3.5]{GK20}.
All in all, there is no loss in generality assuming that $X=X'$ is split, which we do from now on.

Since $\HH1.Y_i.\R.= \HH1.Z_j.\R.=0$, the Künneth decomposition on the space $\HH2.X.\R.$ enables us to write 
\[\beta= p_T^*\beta_T+\sum_{i\in I} p_{Y_i}^*\beta_{Y_i}+ \sum_{j\in J} p_{Z_j}^* \beta_{Z_j}\] 
where $\beta_T,\ \beta_{Y_i}$ and $\beta_{Z_j}$ are \kahler classes on $T,\ Y_i$ and $Z_j$ respectively.
In particular, we get 
\[\ccorb2{X} \cdot \beta^{n-2} = \sum_{i\in I} \lambda_i \, \ccorb2{Y_i}\cdot \beta_{Y_i}^{\dim(Y_i)-2}+ \sum_{j\in J} \mu_j \,\ccorb2{Z_j}\cdot \beta_{Z_j}^{\dim(Z_j)-2}, \]
where $\lambda_i,\, \mu_j>0$ are positive combinatorial coefficients.
\cref{cy} and \cref{ihs} imply that the above quantity is non-negative, and strictly positive unless $I = J = \emptyset$; i.e.~unless $X = T$ is a torus.
\cref{th:MY singular Kahler} and \cref{th:torus quotients} are now proved. \qed

\subsection{Proof of \cref{th:characterization torus quotient}}

To finish, we prove \cref{th:characterization torus quotient} by proving the various implications separately, similar to \cref{main3}.

\subsection*{\lref{torus.1} $\imp$ \lref{torus.2}}

This is \cref{th:torus quotients}.

\subsection*{\lref{torus.2} $\imp$ \lref{torus.3}}

Let $\pi \from \C^n \to (X, \Delta)$ be the universal cover of $(X, \Delta)$.
Endowing the orbifold $(X, \Delta)$ with a metric, we infer that the group $\Gamma \defn \piorb{X,\Delta}$ is then (isomorphic to) a discrete cocompact subgroup of $\C^n \rtimes \mathrm{U}(n)$.
We can then resort to Bieberbach's theorem and exhibit a finite index normal subgroup $\Lambda \subset \Gamma$ that is a lattice $\Lambda \subset \C^n$ acting by translation.
The universal cover factors through the corresponding quotient:
\[ \C^n \xrightarrow{\quad\;\;\quad} \factor{\C^n}{\Lambda} \xrightarrow{\quad f \quad} \factor{\C^n}{\Gamma} = X, \]
where $f$ is the quotient by the action of the finite group $G \defn \factor{\Gamma}{\Lambda}$ on the complex torus $T \defn \factor{\C^n}{\Lambda}$.
Combining this with the fact that $\pi$ is branched exactly at $\Delta$, we infer that $f$ is orbi-\'etale.

\subsection*{\lref{torus.3} $\imp$ \lref{torus.1}}

If $f \from T \to X$ is a Galois orbi-étale map (for the pair $(X, \Delta)$) from a complex torus, the section $(\mathrm d z_1 \wedge \cdots \wedge \mathrm d z_n)^{\tensor m}$ is $G$-invariant, where $G \defn \Gal(f)$ and $m \defn |G|$.
This proves that $m( K_X + \Delta ) \sim 0$ and thus that $\cc1{K_X+\Delta} = 0$.
Let $\omega_T$ be any \kahler metric on $T$ and let us consider
\[ \omega_f \defn \sum_{g\in G} g^*\omega_T. \]
It descends to an orbifold \kahler metric $\omega_X$ on $(X, \Delta)$ and, the map $f$ being orbi-étale, we have:
\[ \ccorb2{X,\Delta} \cdot [\omega_X]^{n-2} = \frac{1}{\deg(f)} \, \ccorb2T \cdot[\omega_f]^{n-2} = 0. \]
Since $[\omega_X]$ is a \kahler class, this ends the proof. \qed

\end{document}